\documentclass[11pt]{amsart}
\usepackage{mathtools}
\usepackage{amsmath}
\usepackage{amssymb}
\usepackage{amsthm}
\usepackage{listings}
\usepackage{float}
\usepackage{braket}
\usepackage{bbm}
\usepackage[a4paper, total={6in, 8in}]{geometry}
\usepackage{comment}
\usepackage{tikz}
\usepackage{pgfplots}
\usepgfplotslibrary{fillbetween}
\pgfplotsset{compat=1.16}
\usepackage{hyperref}
\hypersetup{
    colorlinks=true,
    linkcolor=blue,
    filecolor=magenta,      
    urlcolor=cyan,
    pdftitle={Overleaf Example},
    pdfpagemode=FullScreen,
    }
\newtheorem{theorem}{Theorem}
\newtheorem{lemma}{Lemma}

\theoremstyle{definition}

\theoremstyle{remark}
\newtheorem{remark}{Remark}

\newcommand{\norm}[1]{\left\lVert#1\right\rVert}

\newcommand{\C}{\mathbb{C}}
\newcommand{\R}{\mathbb{R}}
\newcommand{\Q}{\mathbb{Q}}
\newcommand{\N}{\mathbb{N}}
\newcommand{\Z}{\mathbb{Z}}

\newcommand{\SL}{\textnormal{SL}}
\begin{document}
	
	\title{Effective Counting and Spiralling of Lattice Approximates}
	\author{Nathan Hughes}
	\address{Department of Mathematics, University of Exeter, Exeter, EX4 4QF, UK}
	\email{nh477@exeter.ac.uk}
	\maketitle
	\begin{abstract}
		Given $d\geq 2$, we show that the number of approximates $\frac{1}{q}\mathbf{p}\in \Q^d$ of $\mathbf{x}\in\R^d$ satisfying $|q\mathbf{x}-\mathbf{p}|\leq cq^{-\frac{1}{d}}$ with denominator $1\leq q < T$ decays to the asymptotic term $c\text{vol}_d(B_d(0,1))\log T$ with error of order $\left(\log T\right)^{-\frac{1}{2}}\left(\log \log T\right)^\frac{3}{2}\left(\log\log\log T\right)^{\frac{1}{2}+\epsilon}$ for almost all $\mathbf{x}\in\R^d$ and for any $\epsilon >0$. Results with the same order are proven for primitive lattice approximates for all $d\geq 1$ and also for the case of linear forms and affine lattices. These results, especially in the primitive case for $d=1$, are an improvement to the results of Schmidt.
	\end{abstract} 
	\section{Introduction}
	In this paper, we show effective results for counting and spiralling of lattice approximates - namely, we give error rates for these results in Diophantine approximation. For spiralling, these error rates are new results, and new proofs are provided for the counting results, different from those provided by Schmidt \cite{schmidtMETRICALTHEOREMGEOMETRY,schmidtMetricalTheoremDiophantine1960}. In particular, for the primitive case with $d=1$ (see theorem \ref{prim}), our result improves the error rate given by Schmidt. We also give effective results for Diophantine approximation of systems of linear forms and affine lattices.
	\subsection{Approximation of Real Vectors}
	Recall the following corollary to Dirichlet's Approximation Theorem.
	\begin{theorem}[{{Dirichlet \cite{Dirichlet}}}]
		For each $d\geq 1$, there exists a constant $C=C(d)>0$ such that, for all $\mathbf{x}\in\R^d$, there exists infinitely many $(\mathbf{p},q)\in\Z^d\times \N$ satisfying
		\begin{equation*}
			\norm{q\mathbf{x}-\mathbf{p}} < C|q|^{-\frac{1}{d}}
		\end{equation*}
	\end{theorem}
	Here, $\norm{\cdot}$ can be any norm on $\R^d$, so we assume it is the Euclidean norm for simplicity. We might ask how these approximates are distributed with respect to $\mathbf{x}$. More concretely, fix $\mathbf{x}\in\R^d$, let $\mathbb{S}^{d-1}$ be the unit sphere in $\R^d$, and let 
	\begin{equation*}
		\theta_\mathbf{x}:\R^d\times \R\rightarrow \mathbb{S}^{d-1} , \quad \theta_{\mathbf{x}}(\mathbf{p},q) = \frac{q\mathbf{x}-\mathbf{p}}{\norm{q\mathbf{x}-\mathbf{p}}}
	\end{equation*}
	Then $\theta_\mathbf{x}(\mathbf{p},q)$ measures the direction of the approximate $(\mathbf{p},q)$ relative to $\mathbf{x}$; this is referred to as \textit{spiralling of lattice approximates} \cite{athreyaSpiralingApproximationsSpherical2014}. Given measurable $A \subset \mathbb{S}^{d-1}$, we can ask what proportion of the directions of approximates $\theta_\mathbf{x}(\mathbf{p},q)$ lie in $A$, and they showed that it is proportional to the volume of $A$ in $\mathbb{S}^{d-1}$ for almost all $\mathbf{x}\in\R^d$.
	\begin{theorem}[{{Athreya, Ghosh, Tseng \cite[Theorem 1.1]{athreyaSpiralingApproximationsSpherical2014}}}]
		\label{AGT}
		Let $d\geq1$ and $A \subset \mathbb{S}^{d-1}$ be measurable. For $T>1$ and $\mathbf{x} \in \R^d$, define the sets
		\begin{equation*}
			N(\mathbf{x},T) = \#\Big\{(\mathbf{p},q)\in\Z^d\times \N \, \Big| \, 0 < q \leq T, \norm{q\mathbf{x}-\mathbf{p}}< C|q|^{-\frac{1}{d}} \Big\} 
		\end{equation*}
		\begin{equation*}
			N(\mathbf{x},T,A) = \#\Big\{(\mathbf{p},q) \in N(\mathbf{x},T) \, \Big| \, \theta_\mathbf{x}(\mathbf{p},q)\in A\Big\}
		\end{equation*}
		Then for almost every $\mathbf{x} \in \R^d$,
		\begin{equation*}
			\lim_{T\rightarrow\infty}\frac{N(\mathbf{x},T,A)}{N(\mathbf{x},T)} = \textnormal{vol}_{\mathbb{S}^{d-1}}(A)
		\end{equation*}
	\end{theorem}
	This result was shown using the Siegel transform of a lattice, which links Euclidean space to the space of unimodular lattices, allowing the introduction of ergodic techniques (such as Birkhoff's ergodic theorem and Moore's ergodicity theorem) and is another example of the link between homogeneous dynamics and number theory. The purpose of this paper is to give an effective version of theorem \ref{AGT} and related problems in Diophantine approximation.
	\subsection{Linear Forms and Affine Lattices}\label{dani}
	The problem of Diophantine approximation can be generalised to higher dimensions as follows. Let $d,m,n \geq 1$ be natural numbers satisfying $d=m+n$. Given $M \in \text{Mat}_{m,n}(\R)$ and $C>0$, we can ask how many pairs $(\mathbf{p},\mathbf{q})\in \Z^m \times \Z^n$ satisfy 
	\begin{equation}\label{linearforms}
		\norm{M\mathbf{q}-\mathbf{p}}<C\norm{\mathbf{q}}^{-\frac{n}{m}}
	\end{equation} 
	Define 
	\begin{equation*}
		\Lambda_M = \begin{pmatrix}
			\text{Id}_m & M \\
			\mathbf{0} & \text{Id}_n 
		\end{pmatrix}
	\end{equation*} 
	Then we have the equality of sets
	\begin{equation*}
		\Lambda_M\Z^d = \Bigg\{ \begin{pmatrix}
			M\mathbf{q}-\mathbf{p} \\ \mathbf{q}\end{pmatrix} \, \Bigg| \, \mathbf{p} \in\Z^m, \mathbf{q}\in\Z^n
		\Bigg\}
	\end{equation*}	
	Define the area
	\begin{equation*}
		R_{T,c}= \Big\{(\mathbf{x},\mathbf{y})\in\R^m \times \R^n \, \Big| \, \norm{\mathbf{x}}^m \norm{\mathbf{y}}^n \leq c ,\, 1 \leq \norm{\mathbf{y}} < T\Big\} \subset \R^d
	\end{equation*}
	We calculate $\textnormal{vol}_d(R_{T,c}) = cB_mC_n\log T$, where $B_m$ is the $m$-dimensional volume of the unit $m$-ball and $C_n$ is the surface area of the unit sphere $\mathbb{S}^{n-1}$. The number of solutions to (\ref{linearforms}) is then equal to $\#(\Lambda_M\Z^d \cap R_{T,c})$ for some $c>0$. This reformulation of a problem of Diophantine approximation to a problem in homogeneous dynamics is known as Dani's correspondence (see \cite[Section 5.2c]{kleinbockChapter11Dynamics2002} for more details). Let $X_d$ be the set of all unimodular lattices in $\R^d$. Then we may identify $X_d$ with the quotient group $\SL_{d}(\R)/\SL_d(\Z)$ and assign a left-$\SL_d(\R)$-invariant Haar measure $\mu=\mu_d$ onto it, which is unique up to scalar multiplication, induced by the unique measure with the same properties on $\SL_d(\R)$.  Asymptotically, we can count the number of points of a unimodular lattice $\Lambda \in X_d$ contained in the region $R_{T,c}$ accordingly.
	
	\begin{theorem}[Athreya, Parrish, Tseng {{\cite[Theorem 1.5]{athreyaErgodicTheoryDiophantine2016}}}]\label{APT1}
		For all $c>0$ and for $\mu$-a.e. unimodular lattice $\Lambda \in X_d$,
		\begin{equation*}
			\lim_{T\rightarrow \infty} \frac{\#\left( \Lambda \cap R_{T,c}\right)}{\textnormal{vol}_d(R_{T,c})} = 1
		\end{equation*}
	\end{theorem}
	A generalisation of this result for affine unimodular lattices is also possible. Given $\Lambda \in X_d$ and a non-zero vector $\mathbf{\xi} \in \R^d / \Lambda$, $\Lambda+\mathbf{\xi}$ is an affine unimodular lattice and we denote by $Y_d$ the set of all such affine unimodular lattices. Similarly to $X_d$, we may further identify the space $Y_d$ with the quotient space $\SL_d(\R) \ltimes \R^d / \SL_d(\Z)\ltimes \Z^d$ and endow $Y_d$ with a natural Haar measure $\nu=\nu_d$. We have the following theorem related to theorem \ref{APT1}.
	\begin{theorem}[Athreya, Parrish, Tseng {{\cite[Theorem 1.6]{athreyaErgodicTheoryDiophantine2016}}}]\label{APT2}
		For all $c>0$ and for $\nu$-a.e. unimodular lattice $\Lambda \in Y_d$,
		\begin{equation*}
			\lim_{T\rightarrow \infty} \frac{\#\left( \left(\Lambda +\mathbf{\xi}\right)\cap R_{T,c}\right)}{\textnormal{vol}_d(R_{T,c})} = 1
		\end{equation*}
	\end{theorem}
	\subsection{Statement of Results}\label{SoR}
	In the following theorems, $f(x)=o_\gamma\left(g(x)\right)$ means that there exists a constant $C(\gamma)>0$ and $x_0\geq 0$ such that for all $x>x_0$, $|f(x)|< C(\gamma)\left(g(x)\right)$. Additionally, define
	\begin{gather*}
		P_{T,c} = \left\{ \left(\mathbf{v}_1,v_2\right)\in\R^d\times \R\, \bigg| \, \norm{\mathbf{v}_1}v_2\leq c , \, 1<v_2\leq T \right\} \subset \R^{d+1}\\
		P_{T,c,A} = \left\{ \left(\mathbf{v}_1,v_2\right)\in P_{T,c}\, \bigg| \, \frac{\mathbf{v}_1}{\norm{\mathbf{v}_1}}\in A \right\}
	\end{gather*}
	Our first result is an effective version of theorem \ref{AGT} for $d\geq 3$.
	\begin{theorem}
		\label{main}
		Let $d\geq 3$ and $A\subset \mathbb{S}^{d-2}$ be measurable. Then for any $\epsilon >0$, for $\mu$-a.e. $\Lambda \in X_{d}$ and for sufficiently large $T$,
		\begin{equation*}
			\frac{\#(\Lambda\cap P_{T,c,A})}{\#(\Lambda \cap P_{T,c})}=\textnormal{vol}_{\mathbb{S}^{d-2}}(A) +o_{A,c,\Lambda,d}\big( (\log T)^{-\frac{1}{2}}(\log \log T)^\frac{3}{2}(\log\log\log T)^{\frac{1}{2}+\epsilon}\big)
		\end{equation*}
	\end{theorem}
	A similar spiralling result can be obtained for all $d\geq2$ by restricting the lattice to only its primitive elements. Given $\Lambda \in X_d$, define $\Lambda^\text{pr}$ to be the set of all points $\mathbf{v}\in \Lambda$ such that $\alpha\mathbf{v}\notin \Lambda$ for all non-zero $|\alpha| \neq 1$.
	\begin{theorem}\label{prim}
		Let $d\geq 2$, and let $A\subset \mathbb{S}^{d-2}$ be measurable. Then for any $\epsilon >0$, for $\mu$-a.e. $\Lambda \in X_{d}$ and for sufficiently large $T$,
		\begin{equation*}
			\frac{\#(\Lambda^{\textnormal{pr}}\cap P_{T,c,A})}{\#(\Lambda^{\textnormal{pr}} \cap P_{T,c})}=\textnormal{vol}_{\mathbb{S}^{d-2}}(A) +o_{A,c,\Lambda,d}\big( (\log T)^{-\frac{1}{2}}(\log \log T)^\frac{3}{2}(\log\log\log T)^{\frac{1}{2}+\epsilon}\big)
		\end{equation*}
	\end{theorem}
	For linear forms, we have the following effective version of theorem \ref{APT1}.
	\begin{theorem}\label{effective2}
		Let $d,m,n\geq 1$ be natural numbers satisfying $d=m+n$. Then for all $\epsilon>0$, for $\mu$-a.e. $\Lambda \in X_d$, and for sufficiently large $T$,
		\begin{equation*}
			\#\left( \Lambda \cap R_{T,c}\right) = cB_mC_n\log T + o_{c,\Lambda,d}\left(\left(\log T\right)^{-\frac{1}{2}} \left(\log \log T\right)^\frac{3}{2} \left(\log\log\log T\right)^{\frac{1}{2}+\epsilon}\right) 
		\end{equation*}
	\end{theorem}
	Finally, for affine lattices, we have the following effective version of theorem \ref{APT2}.
	\begin{theorem}\label{effective3}
		Let $d,m,n\geq 1$ be natural numbers satisfying $d=m+n$. Then for all $\epsilon>0$, for $\nu$-a.e. $\Lambda+\mathbf{\xi} \in Y_d$ and for sufficiently large $T$,
		\begin{equation*}
			\#\left(  \left(\Lambda+\mathbf{\xi}\right) \cap  R_{T,c} \right) = cB_mC_n\log T + o_{c,\Lambda,\mathbf{\xi},d}\left(\left(\log T\right)^{-\frac{1}{2}} \left(\log \log T\right)^\frac{3}{2} \left(\log\log\log T\right)^{\frac{1}{2}+\epsilon}\right) 
		\end{equation*}
	\end{theorem}
	\subsection{Structure of Paper}
	In section \ref{meanvaluetheorems}, we provide necessary background on the homogeneous spaces of unimodular lattices and integrals of some functions on these spaces. We employ the use of a theorem of Gaposhkin (theorem \ref{Gaposhkin}) on the relationship between the second moment of an ergodic sum and the convergence of said sum, introduced in section \ref{ergThm}. In most cases, Rogers' theorem (theorem \ref{RMVT}) is suitable to calculate the second moment, but this theorem fails to hold on $X_2$. Section \ref{2d} is dedicated to using a theorem of Kleinbock and Yu (theorem \ref{KleinbockYu}) to calculate the second moment in the primitive two-dimensional case. In section \ref{further}, the results for counting lattice approximates for systems of linear forms and affine lattices are given, the latter using a result of El-Baz, Marklof and Vinogradov (lemma \ref{EBMV}).
	\section{The Homogeneous Space of Unimodular Lattices}\label{meanvaluetheorems}
	\subsection{Mean Value Theorems on the Space of Unimodular Lattices}
	Let $d\geq 2$, $G_d=\SL_{d}(\R)$, $\Gamma_d=\SL_{d}(\Z)$ and consider the quotient space $G_d/\Gamma_d$. This space can be identified with the homogeneous space $X_d$ via the mapping $g\Gamma_d \mapsto g\Z^{d}$, where $g\in G_d$. As shown in section \ref{dani}, a vector $\mathbf{x}\in\R^{d-1}$ can be identified with a lattice $\Lambda_\mathbf{x} \in X_{d}$.
	
	For $d\geq 2$, let $f:\R^{d} \rightarrow \R$ be a Riemann-integrable with compact support and $\Lambda \in X_d$. The \textit{Siegel transform} of $f$, denoted $\widehat{f}$, is defined as
	\begin{equation*}
		\widehat{f}(\Lambda) = \sum_{\mathbf{v}\in\Lambda \setminus \{\mathbf{0}\}}f(\mathbf{v})
	\end{equation*}
	A theorem of Siegel shows sufficient conditions for a Siegel transform to be integrable and provides a formula to calculate said integral. Let $d\mathbf{x}$ denote the Lebesgue measure on $\R^d$.
	\begin{theorem}[{{\cite[Siegel]{siegelMeanValueTheorem1945}, see also \cite{macbeathSiegelMeanValue1958}}}]\label{SMVT}
		Let $f \in L^1(\R^{d},d\mathbf{x})$. Then $\widehat{f} \in L^1(X_d,\mu)$ and
		\begin{equation*}
			\int_{\R^d} f(\mathbf{x})\,d\mathbf{x} = \int_{X_d} \widehat{f}(\Lambda)\,d\mu(\Lambda)
		\end{equation*}
	\end{theorem}
	For our purposes, observe that for any measurable set $B \subset \R^d\setminus \{\mathbf{0}\}$,
	\begin{equation*}
		\int_{X_d} \widehat{\mathbbm{1}}_B(\Lambda)\,d\mu(\Lambda) = \int_{\R^d}\mathbbm{1}_B(\mathbf{x})\,d\mathbf{x} = \text{vol}_d(B)
	\end{equation*}
	Theorem \ref{SMVT} was partially generalised by Rogers, providing an equation for more general functions on $X_d$, where $d\geq 3$. 
	\begin{theorem}[{{Rogers \cite[Theorem 4]{rogersMeanValuesSpace1955} \cite{macbeathSiegelMeanValue1958}}}]
		\label{RMVT}
		Let $d \geq 3$, and suppose $\rho:(\R^d)^k \rightarrow \R$ is non-negative and Borel measurable, for some $1 \leq k \leq d-1$. Then
		\begin{equation}
			\label{Rogers}
			\begin{gathered}
				\int_{X_d} \sum_{x_1,\dots,x_k \in \Lambda} \rho(x_1,\dots,x_k)\,d\mu(\Lambda) = \rho(\mathbf{0}) + \int_{\R^d}\cdots \int_{\R^d} \rho(x_1,\cdots,x_k)\,dx_1\cdots dx_k + \\
				\sum_{(\nu;\mu) \in \mathcal{S}} \sum_{q=1}^\infty \sum_{D \in \Phi(\mathcal{S},q)} \bigg(\frac{e_1}{q}\cdots\frac{e_m}{q}\bigg)^d \int_{\R^d}\cdots \int_{\R^d} \rho\Bigg( \sum_{i=1}^m \frac{d_{i1}}{q}x_i , \dots, \sum_{i=1}^m \frac{d_{ik}}{q}x_i \Bigg)\, dx_1\cdots dx_m
			\end{gathered}
		\end{equation}
		where $\mathcal{S}$ is the set of all partitions of $\{1,\dots, k\}$ into two increasing sequences $(\nu_1,\cdots,\nu_m)$, $(\mu_1,\dots,\mu_{m-k})$ for some $1 \leq m \leq k-1$, where $\Phi((\nu;\mu),q)$ is the set of all $D \in \textnormal{Mat}_{m,k}(\Z)$ with elements having highest common factor relatively prime to $q$, with
		\begin{gather*}
			d_{i{\nu_{j}}} = q\delta_{ij} \\
			d_{i{\mu_j}} = 0 \quad \text{if  } \mu_j < \nu_i
		\end{gather*}
		for $1 \leq i \leq m$, $1 \leq j \leq k-m$, and where $e_i = \gcd(\epsilon_i,q)$ for $1\leq i \leq m$, for elementary divisors $\epsilon_i$ of the matrix $D \in \Phi((\nu;\mu),q)$ and $\delta_{ij}$ is the Kronecker delta function.
	\end{theorem}
	In particular, this theorem allows us to calculate the integral of $\widehat{f}^k$ over $X_d$, for any $d\geq 3$ and $1 \leq k \leq d-1$.
	\subsection{Second Moments of Characteristic Functions on the Space of Unimodular Lattices}\label{setup}
	Suppose $(X,\mu)$ is a probability space and let $\phi$ be an ergodic transformation on $X$. For $f \in L^2(X,\mu)$, the $n^\text{th}$ correlation coefficient is defined as
	\begin{equation*}
		b_n(f,\phi) = \int_X (f\circ \phi^n)f\, d\mu 
	\end{equation*}
	We will calculate the decay of these correlations and employ theorem \ref{Gaposhkin} to obtain the rate of convergence of an ergodic sum. Notice that 
	\begin{equation*}
		\norm{\sum_{k=0}^{N-1} f\circ\phi^k}_2^2 = Nb_0(f,\phi)+2\sum_{k=1}^{N-1} (N-k)b_k(f,\phi) \in O(N^2 b_N(f,\phi))
	\end{equation*}
	so it suffices to calculate the asymptotic behaviour in $n$ of the second moment of an ergodic sum. 
	\begin{lemma}
		\label{lemma}
		Let $d\geq 3$ and $B \subset \R^{d}\setminus\{0\}$ be measurable. Then 
		\begin{equation*}
			\int_{X_{d}} \widehat{\mathbbm{1}}_{B}(\Lambda)^2 \,d\mu(\Lambda) \leq \left(\textnormal{vol}_d(B) \right)^2 + O_d\big(\textnormal{vol}_d(B)\big)
		\end{equation*}
		\begin{proof}
			We apply theorem \ref{RMVT} with $k=2$ and $\rho(x,y) = \mathbbm{1}_{B}(x)\mathbbm{1}_B(y)$. Note that, since $\mathbf{0}\notin B$,
			\begin{equation*}
				\sum_{x,y \in \Lambda} \mathbbm{1}_{B}(x) \mathbbm{1}_{B}(y) = \bigg(\sum_{x \in \Lambda} \mathbbm{1}_{B}(x)\bigg)\bigg( \sum_{y \in \Lambda} \mathbbm{1}_{B}(y)\bigg) =  \widehat{ \mathbbm{1}}_B(\Lambda) ^2
			\end{equation*}Since $k=2$, we have $m=1$ and $\mathcal{S} = \{ (1;2), (2;1) \}$. For each element of $\mathcal{S}$, any matrix $D$ is of size $1\times 2$ with coprime entries. For any $q \in \N$, if  $D \in \Phi((1;2),q)$ then $D = (q,r)$ for some $r\in\Z$ with $\gcd(q,r) =1$. To compute the elementary divisor of $D$, note that equivalent matrices have the same elementary divisors \cite{newmanSmithNormalForm1997}. Since $\gcd(q,r) =1$, $D$ is equivalent to $(1,0)$, so $1$ is the elementary divisor of $D$.
			Suppose now that $D \in \Phi((2;1),q)$. Writing $D=(d_{11},d_{22})$, we have $d_{12}=q$ and $d_{11}=0$, but $\gcd(\gcd(0,q),q)=q$, so $\Phi((2;1),q)$ is empty. We can bound equation \ref{Rogers} by 
			\begin{align*}
				\int_{X_{d}} \widehat{\mathbbm{1}}_B(\Lambda)^2\,d\mu(\Lambda) &=  \Bigg(\int_{\R^{d}} \mathbbm{1}_B(x)\,dx \Bigg)^2+ 
				\sum_{q=1}^\infty \sum_{\substack{r \in \Z \\ \gcd(q,r)=1}} \bigg(\frac{1}{q}\bigg)^{d}\int_{\R^{d}} {\mathbbm{1}}_B( x){\mathbbm{1}}_B\Big(\frac{rx}{q}\Big) \,dx \\
				&=  \left(\textnormal{vol}_d(B)\right)^2+
				\sum_{q=1}^\infty \sum_{\substack{r \in \Z \\ \gcd(q,r)=1}} \int_{\R^{d}} {\mathbbm{1}}_B( qx){\mathbbm{1}}_B(rx) \,dx \\
				&\leq \left(\textnormal{vol}_d(B)\right)^2 + \sum_{q=1}^\infty \sum_{r\in \Z\setminus\{0\}} \int_{\R^{d}} \mathbbm{1}_B(qx)\mathbbm{1}_B(rx)\,dx \\
			\end{align*}
			Finally, using H\"older's inequality, 
			\begin{gather*}
				\sum_{q=1}^\infty \sum_{r\in \Z\setminus\{0\}} \int_{\R^{d}} \mathbbm{1}_B(qx)\mathbbm{1}_B(rx)\,dx\\\leq  \sum_{q=1}^\infty \Bigg( \int_{\R^{d}} \mathbbm{1}_{B}(qx) \,dx\Bigg)^\frac{1}{2}\sum_{r\in \Z\setminus\{0\}} \Bigg( \int_{\R^{d}} \mathbbm{1}_{B}(rx) \,dx\Bigg)^\frac{1}{2} \\
				= \sum_{q=1}^\infty \Bigg( \frac{1}{q^{d}}\int_{\R^{d}} \mathbbm{1}_{B}(x)\,dx \Bigg)^\frac{1}{2}\sum_{r\in \Z\setminus\{0\}}  \Bigg( \frac{1}{|r|^{d}}\int_{\R^{d}} \mathbbm{1}_{B}(x)\,dx \Bigg)^\frac{1}{2}  \\
				= 2\zeta\Big(\frac{d}{2}\Big)^2 \text{vol}_d(B)
			\end{gather*}
		\end{proof}
	\end{lemma}
	\section{Effective Ergodic Theorems}\label{ergThm}
	\subsection{Effective Theorem for Siegel Transforms}
	Our method of proving theorems \ref{main}, \ref{prim}, \ref{effective2} and \ref{effective3} uses the following result derived from a theorem of Gaposhkin.
	\begin{theorem}[{{Gaposhkin \cite[Theorem 3(iv)]{gaposhkinDependenceConvergenceRate1982}, Kachurovskii \cite[Theorem 15(ii)]{kachurovksiiRateCongergenceErgodic}}}]
		\label{Gaposhkin}
		Let $f\in L^2(X,\mu)$ and let $\phi$ be an ergodic transformation of $X$. If
		\begin{equation}
			\label{kaccond}
			\int_{X} \Bigg( \sum_{n=0}^{N-1}\bigg(f\circ \phi^n - \int_X f\,d\mu\bigg) \Bigg)^2 \, d\mu \in O_{f,\phi}(N) 
		\end{equation}
		Then for any $\epsilon >0$, and $\mu$-a.e. $x \in X$,
		\begin{equation}
			\label{effective}
			\frac{1}{N} \sum_{n=0}^{N-1} f\big( \phi^n(x)\big) = \int_X f \,d\mu + o_{f,\phi}\big(N^{-\frac{1}{2}}(\log N)^{\frac{3}{2}} (\log \log N)^{\frac{1}{2}+\epsilon}\big)
		\end{equation}
	\end{theorem}
	This theorem was originally stated by Gaposhkin for mean-square continuous wide-sense stationary random processes, and later stated by Kachurovskii for functions $f\in L^2(X,\mu)$ satisfying $\int_Xf d\mu =0$. In the context of homogeneous spaces, this theorem has been used by Kleinbock, Shi and Weiss to give error terms for the convergence of the ergodic integral
	\begin{equation*}
	\frac{1}{T}\int_0^T \phi\left(g_t u(\vartheta)\Lambda \right)dt = \int_X \phi d\mu + o\left(T^{-\frac{1}{2}}\log^{\frac{3}{2}+\epsilon} T \right)
	\end{equation*}
	for any smooth, compactly supported functions on $X_d$ for $d\geq 2$ and for almost every $\vartheta \in \text{Mat}_{m,n}(\R)$ (to be contrasted with the measure used in theorem \ref{effective2}), leading to a generalisation of theorem \ref{APT1} in the non-equal weight case  \cite{kleinbockPointwiseEquidistributionError2017}.
	\begin{proof}[Proof of theorem \ref{main}]
		For $t \in \R$, define 
		\begin{equation*}
			g_t = \begin{pmatrix}
				e^t \text{Id}_{d-1} & \mathbf{0} \\ \mathbf{0}^\textnormal{T} & e^{-t(d-1)}
			\end{pmatrix}
		\end{equation*}
		Then $\{g_t\}_{t\in\R}$ defines a flow on $X_{d}$ via left translation: $g_t\cdot g\Gamma_d = (g_tg)\Gamma_d$. By Moore's Ergodicity Theorem \cite[Theorem 3]{mooreErgodicityFlowsHomogeneous1966} (see also \cite[\S3.2 Theorem 2.1]{bekkaErgodicTheoryTopological2000}), since the closure of $\{g_t\}_{t\in\R}$ is non-compact in $G_d$, this flow is ergodic.
		
		Let $T>1$, $N \in \N$ and set $s = \frac{\log T}{d-1}$. For any distinct integers $n>m\geq 1$, $g_s^n P_{T,c} \cap g_s^m P_{T,c} = \emptyset$ and 
		\begin{equation*}
			\bigcup_{k=m}^n g_s^{-k}P_{T,c} = P_{T^{n+1},c}\setminus P_{T^m,c}
		\end{equation*}
		In particular, we have
		\begin{gather*}
			\bigcup_{k=0}^{N-1} g_s^{-k} P_{T,c} = P_{T^N,c}\\ \sum_{k=0}^{N-1} \widehat{\mathbbm{1}}_{P_{T,c}} (g_s^k\Lambda) = \widehat{\mathbbm{1}}_{P_{T^N,c}}(\Lambda) 
		\end{gather*}
		By Siegel's Mean Value Theorem (Theorem \ref{SMVT}), the function $\widehat{\mathbbm{1}}_{P_{T,c}}(g_s^k\Lambda) - \text{vol}_d(P_{T,c})$ is integrable for all $k\in\N$, and moreover,
		\begin{equation*}
			\int_{X_d}\widehat{\mathbbm{1}}_{P_{T,c}}(g_s^k\Lambda) - \text{vol}_d(P_{T,c})\,d\mu = 0
		\end{equation*}
		Rogers' theorem (theorem \ref{RMVT}) implies that $\left(\widehat{\mathbbm{1}}_{P_{T,c}}(g_s^k\Lambda) - \text{vol}_d(P_{T,c})\right)^2$ is integrable, so we can calculate 
		\begin{gather*}
			\int_{X_{d}} \Bigg( \sum_{k=0}^{N-1} \widehat{\mathbbm{1}}_{P_{T,c}} (g_s^k\Lambda) - N\text{vol}_d(P_{T,c}) \Bigg)^2\,d\mu \\ = \int_{X_{d}} \Bigg[\Bigg( \sum_{k=0}^{N-1} \widehat{\mathbbm{1}}_{P_{T,c}} (g_s^k\Lambda)\Bigg)^2 - 2N\text{vol}_d(P_{T,c})  \sum_{k=0}^{N-1} \widehat{\mathbbm{1}}_{P_{T,c}} (g_s^k\Lambda) + N^2\text{vol}_d(P_{T,c})^2 \Bigg]\,d\mu \\
			= \int_{X_{d}} \Bigg( \sum_{k=0}^{N-1} \widehat{\mathbbm{1}}_{P_{T,c}} (g_s^k\Lambda)\Bigg)^2 \, d\mu- N^2\text{vol}_d(P_{T,c})^2
		\end{gather*}
		Lemma \ref{lemma} implies that this can be bounded as
		\begin{align*}
			&\int_{X_{d}} \Bigg( \sum_{k=0}^{N-1} \widehat{\mathbbm{1}}_{P_{T,c}} (g_s^k\Lambda)\Bigg)^2 \, d\mu- N^2\text{vol}_d(P_{T,c})^2 \\\leq& \Bigg(\int_{\R^{d}} \sum_{k=0}^{N-1}\mathbbm{1}_{P_{T,c}}(\mathbf{x})\, d\mathbf{x} \Bigg)^2 + O_d\big(N\text{vol}_d(P_{T,c})\big) - N^2\text{vol}_d(P_{T,c})^2 \\ =& O_{T,c,d}(N)
		\end{align*}
		This shows condition (\ref{kaccond}) in theorem \ref{Gaposhkin} to be true, so for any $\epsilon >0$,
		\begin{equation} \label{thiscanvary}
			\frac{1}{N}\sum_{k=0}^{N-1}\widehat{\mathbbm{1}}_{P_{T,c}}(g_s^k\Lambda) = \frac{1}{N}\widehat{\mathbbm{1}}_{P_{T^N,c}}(\Lambda) = \text{vol}_d(P_{T,c}) + o_{T,c,\Lambda,N,d}\big(N^{-\frac{1}{2}}\left(\log N\right)^\frac{3}{2}(\log \log N)^{\frac{1}{2}+\epsilon}\big)
		\end{equation}
		Given any $\tau >2$ there exists unique $T\in[2,4)$ and $N \in \N$ with $\tau=T^N$. Write $\tau = T^N >1$. Then for any $\epsilon>0$,
		\begin{equation}
			\frac{1}{\log \tau} \widehat{\mathbbm{1}}_{P_{\tau,c}}(\Lambda) = cB_d + o_{T,c,\Lambda,d}\Big( (\log \tau)^{-\frac{1}{2}}(\log \log \tau)^\frac{3}{2}(\log\log\log\tau)^{\frac{1}{2}+\epsilon}\Big)
		\end{equation}
		Similarly, for any measurable $A \subset \mathbb{S}^{d-2}$, 
		\begin{equation*}
			\frac{1}{\log \tau} \widehat{\mathbbm{1}}_{P_{\tau,c,A}}(\Lambda) =cB_d\text{vol}_{\mathbb{S}^{d-2}}(A) + o_{A,T,c,\Lambda,d}\Big( (\log \tau)^{-\frac{1}{2}}(\log \log \tau)^\frac{3}{2}(\log\log\log\tau)^{\frac{1}{2}+\epsilon}\Big)
		\end{equation*}
		By varying $T$, we conclude
		\begin{equation}
			\label{ssss}
			\frac{\widehat{\mathbbm{1}}_{P_{\tau,c,A}}(\Lambda)}{\widehat{\mathbbm{1}}_{P_{\tau,c}}(\Lambda)} = \text{vol}_{\mathbb{S}^{d-2}}(A)  +o_{A,c,\Lambda,d}\big( (\log \tau)^{-\frac{1}{2}}(\log \log \tau)^\frac{3}{2}(\log\log\log\tau)^{\frac{1}{2}+\epsilon}\big)
		\end{equation}
	\end{proof}
	\begin{remark}
		In 1960, Schmidt studied a similar counting problem for lattice points of a (not necessarily unimodular) lattice $\Lambda$ contained in a family of nested finite measurable sets $\Phi \ni S$ which contain arbitrarily large volumes \cite{schmidtMETRICALTHEOREMGEOMETRY}. He found that in dimensions $3$ and greater, for any non-decreasing $\psi:[0,\infty)\rightarrow (0,\infty)$ such that $\int_0^\infty\psi(s)^{-1}ds$ exists, that
		\begin{equation*}
			\#(\Lambda \cap S) = \text{vol}_d(S) + O(\text{vol}_d(S)^{\frac{1}{2}} \log \text{vol}_d(S) \psi^\frac{1}{2}(\log \text{vol}_d(S)))
		\end{equation*}
		Noting that $\text{vol}_d(P_{\tau,c,A}) = c \text{vol}_d\big(B_d\big)\text{vol}_d(A) \log \tau$, we may rewrite (\ref{ssss}) as
		\begin{equation}
			\label{newphi}
			\#(\Lambda \cap P_{\tau,c,A}) = \text{vol}_d(P_{\tau,c,A}) + o_{\Lambda,c,d}\big( (\log \tau)^\frac{1}{2}(\log\log \tau)^\frac{3}{2} (\log \log \tau)^{\frac{1}{2}+\epsilon} \big)
		\end{equation}
		Then setting $\psi$ to be $f(s)=(\log s)^3 (\log\log s)^{1+2\epsilon}$ in Schmidt's theorem would give (\ref{newphi}), but notice that $\int_1^\infty (f(s))^{-1} ds$ does not converge. The result of theorem \ref{main} is therefore a slight improvement to Schmidt's result in the specific case of $\Phi=\{P_{T,c,A}\,|\,T>1\}$ for any $c>0$ and $A \subset \mathbb{S}^{d-2}$. Additionally, Schmidt found that, in two dimensions,
		\begin{equation*}
			\#(\Lambda \cap S) = \text{vol}_2(S) + O(\text{vol}_2(S)^{\frac{1}{2}} \log^2 \text{vol}_2(S) \psi^\frac{1}{2}(\log \text{vol}_2(S)))
		\end{equation*}
		We will see in section \ref{2d} that, in the case of $\Phi=\{P_{T,c,A}\,|\,T>1\}$, this can be improved to the same bound as in the higher-dimensional case.
	\end{remark}
	\begin{remark}
		As Gaposhkin states in his work, any number of iterated logarithms of the form $(N)^{-\frac{1}{2}} ( \log N)^\frac{3}{2} \cdots (\log \cdots \log N)^{\frac{1}{2}+\epsilon}$ may be used in \ref{thiscanvary} \cite[\S 3]{gaposhkinDependenceConvergenceRate1982}. In fact, any function $\Psi(t)$ that satisfies 
		\begin{equation*}
			\int_1^\infty \frac{1}{t\Psi(t)}\,dt < \infty
		\end{equation*}
		can be used to derive an error term. In \ref{effective}, the function
		\begin{equation*}
			\Psi(t) =\begin{cases}
				1 & 1 \leq t \leq e+\delta \\
				(\log  t) (\log \log t)^{1+\epsilon} &  t> e + \delta 
			\end{cases}
		\end{equation*}
		was used, for any $\delta >0$. Given such a function $\Psi$, the effective term will be of the form
		\begin{equation*}
			o\bigg( \sqrt{\frac{\Psi(T)}{T}} \log T \bigg)
		\end{equation*}
		(see Lemma 7 in \cite{gaposhkinDependenceConvergenceRate1982}).
	\end{remark}
	\subsection{Effective Theorem for Primitive Siegel Transforms}\label{primthm}
	Recall that the primitive lattice $\Lambda^\text{pr}$ is the set of all points $\mathbf{v}\in \Lambda$ such that $\alpha\mathbf{v}\notin \Lambda$ for all non-zero $|\alpha| \neq 1$. Define the primitive Siegel transform of a function $f:\R^d \rightarrow \R$ by 
	\begin{equation*}
		\widetilde{f}:X_d \rightarrow \R \, , \quad \widetilde{f}(\Lambda) = \sum_{\text{v}\in\Lambda^\text{pr}} f(\mathbf{v})
	\end{equation*}
	We are interested in obtaining similar theorems to those obtained above with this different transform. Siegel provided the following formula for the primitive Siegel transform.
	\begin{theorem}[{{\cite[Siegel]{siegelMeanValueTheorem1945}, see also \cite{macbeathSiegelMeanValue1958}}}]\label{PSMVT}
			Let $f \in L^1(\R^{d},d\mathbf{x})$. Then $\widetilde{f} \in L^1(X_d,\mu)$ and
			\begin{equation*}
				\frac{1}{\zeta(d)}\int_{\R^d} f(\mathbf{x})\,d\mathbf{x} = \int_{X_d} \widetilde{f}(\Lambda)\,d\mu(\Lambda)
			\end{equation*}
		\end{theorem}
	In the same paper that \cite[Theorem 4]{rogersMeanValuesSpace1955} was proved, Rogers also provided a moment formula for the primitive Siegel transform.
	\begin{theorem}[{{\cite[Theorem 5]{rogersMeanValuesSpace1955}}}] \label{Rogerprim}
		Let $d\geq 3$ and suppose $f:X_{d}\rightarrow \R_{\geq 0}$ is measurable. Then
		\begin{align*}
			\int_{X_{d}} \widetilde{f}(\Lambda)\, d\mu(\Lambda) =& \left(\frac{1}{\zeta(d)} \int_{\R^{d}} f(\mathbf{x})\, d\mathbf{x}\right)^2 + \frac{1}{\zeta(d)} \int_{\R^{d}} \left(f(\mathbf{x})\right)^2\,d\mathbf{x}\\&+ \frac{1}{\zeta(d)}\int_{\R^{d}} f(\mathbf{x})f(-\mathbf{x})\,d\mathbf{x}
		\end{align*}
	\end{theorem}
	\begin{proof}[Proof of theorem \ref{prim} for $d\geq 3$]
		We have, for $s=\frac{\log T}{d-1}$, and using theorem \ref{PSMVT},
		\begin{align*}
			&\int_{X_{d}} \left( \sum_{k=0}^{N-1}\widetilde{\mathbbm{1}}_{P_{T,c}}(g_s^k \Lambda) - \frac{N\textnormal{vol}(P_{T,c})}{\zeta(d)} \right)^2\, d\mu \\
			=&\int_{X_{d}} \left( \sum_{k=0}^{N-1}\widetilde{\mathbbm{1}}_{P_{T,c}}(g_s^k \Lambda)\right)^2 \,d\mu - \frac{2N\textnormal{vol}(P_{T,c})}{\zeta(d)} \int_{X_{d}}\sum_{k=0}^{N-1}\widetilde{\mathbbm{1}}_{P_{T,c}}(g_s^k \Lambda)\,d\mu + \left( \frac{N\textnormal{vol}(P_{T,c})}{\zeta(d)} \right)^2 \\
			=& \int_{X_{d}} \left( \sum_{k=0}^{N-1}\widetilde{\mathbbm{1}}_{P_{T,c}}(g_s^k \Lambda)\right)^2 \,d\mu - \left( \frac{N\textnormal{vol}(P_{T,c})}{\zeta(d)} \right)^2 \\
			=& \int_{X_{d}}\left( \widetilde{\mathbbm{1}}_{P_{T^N,c}}(\Lambda)\right)^2\, d\mu - \left( \frac{\textnormal{vol}(P_{T^N,c})}{\zeta(d)} \right)^2
		\end{align*}
		Using Theorem \ref{Rogerprim}, 
		\begin{align*}
			& \int_{X_{d}}\left( \widetilde{\mathbbm{1}}_{P_{T^N,c}}(\Lambda)\right)^2 \,d\mu - \left( \frac{\textnormal{vol}(P_{T^N,c})}{\zeta(d)} \right)^2 \\
			=& \frac{1}{\zeta(d)^2} \left(\int_{\R^{d}} {\mathbbm{1}}_{P_{T^N,c}}(\mathbf{x})\,d\mathbf{x} \right)^2 + \frac{1}{\zeta(d)} \int_{\R^{d}} ( {\mathbbm{1}}_{P_{T^N,c}}(\mathbf{x}))^2\,d\mathbf{x} \\&+ \frac{1}{\zeta(d)} \int_{\R^{d}}  {\mathbbm{1}}_{P_{T^N,c}}(\mathbf{x}) {\mathbbm{1}}_{P_{T^N,c}}(-\mathbf{x}) \,d\mathbf{x} - \left( \frac{\textnormal{vol}(P_{T^N,c})}{\zeta(d)} \right)^2\\
			=& \frac{\text{vol}(P_{T^N,c})^2}{\zeta(d)^2} + \frac{\text{vol}(P_{T^N,c})}{\zeta(d)}- \left( \frac{\textnormal{vol}(P_{T^N,c})}{\zeta(d)} \right)^2 \\
			=& \frac{\text{vol}(P_{T^N,c})}{\zeta(d)}\\
			=& O_{T,c,d}\left( N \right)
		\end{align*}
		A similar argument to that of theorem \ref{main} yields, for any $A \subset \mathbb{S}^{d-2}$ with measurable boundary,
		\begin{equation}\label{primcount}
			\widetilde{ \mathbbm{1}}_{P_{T,c,A}}(\Lambda) = cB_{d-1}\text{vol}(A)\log T + +o_{A,c,\Lambda,d}\big( (\log T)^{-\frac{1}{2}}(\log \log T)^\frac{3}{2}(\log\log\log T)^{\frac{1}{2}+\epsilon}\big)
		\end{equation}
		\begin{equation}\label{primsprial}
			\frac{\widetilde{\mathbbm{1}}_{P_{T,c,A}}(\Lambda)}{\widetilde{\mathbbm{1}}_{P_{T,c,A}}(\Lambda)} = \text{vol}(A)  +o_{A,c,\Lambda,d}\big( (\log T)^{-\frac{1}{2}}(\log \log T)^\frac{3}{2}(\log\log\log T)^{\frac{1}{2}+\epsilon}\big)
		\end{equation}
	\end{proof}
	\section{The Two-Dimensional Case}\label{2d}
	\subsection{Outline of Proof}
	We wish to find a bound, asymptotic in $T$, for the expression
	\begin{align*}
		&\int_{X_2} \left( \widetilde{\mathbbm{1}}_{P_{T,c}}(\Lambda) - \frac{\textnormal{area}(P_{T,c})}{\zeta(2)}\right)^2\, d\mu(\Lambda) \\
		=& \norm{\widetilde{\mathbbm{1}}_{P_{T,c}}}_2^2 - 2\frac{\textnormal{area}(P_{T,c})}{\zeta(2)} \int_{X_{2}}\widetilde{\mathbbm{1}}_{P_{T,c}}(\Lambda) \, d\mu(\Lambda) + \frac{\textnormal{area}(P_{T,c})^2}{\zeta(2)^2} \\
		=& \norm{\widetilde{\mathbbm{1}}_{P_{T,c}}}_2^2 - \frac{4c^2}{\zeta(2)^2}\log(T)
	\end{align*}
	Rogers' theorems (theorems \ref{RMVT} and \ref{Rogerprim}) cannot be used to calculate this expression in either the full or primitive case of the Siegel transform, since these theorems do not hold in the case $d=2$. In its stead, we will use the following theorem. 
	\begin{theorem}[{{\cite[Theorem 2.1]{kleinbockDynamicalBorelCantelliLemma2020}}}]\label{KleinbockYu}
		Let $\mathcal{S}$ be a measurable and bounded subset of $\R^2$, let $f = \mathbbm{1}_\mathcal{S}$, and let $-{\mathcal{S}} = \{\mathbf{x} \in \R^2\, | \, -\mathbf{x}\in\mathcal{S}\}$. Then
		\begin{equation*}
			\norm{\widetilde{f}}_2^2 = \frac{1}{\zeta(2)} \left(\textnormal{area}(\mathcal{S}) + \textnormal{area}(\mathcal{S}\cap -{\mathcal{S}}) + \sum_{n\neq 0}\frac{\varphi(|n|)}{|n|} \iint\limits_{\mathcal{S}}\left| \mathcal{I}_{(x,y)}^n(\mathcal{S}) \right| \,dx\,dy\right)
		\end{equation*}
		where
		\begin{equation*}
			\left|\mathcal{I}_{(x,y)}^n(\mathcal{S}) \right| = \left\{ t \in \R \, \Bigg| \, \frac{n}{x^2+y^2}(-y,x) + t(x,y) \in \mathcal{S} \right\}.
		\end{equation*}
	\end{theorem}
	Using this theorem, we will show that $\widetilde{\mathbbm{1}}_{P_{T,c}}\in L^2(X_2,\mu)$, in which case we may apply theorem \ref{Gaposhkin}. To calculate $\norm{\widetilde{\mathbbm{1}}_{P_{T,c}}}_2^2$, we first find an expression for $|\mathcal{I}_{(x,y)}^n(P_{T,c})|$ in terms of the y-coordinates of the intersections of the line $\left\{\left(\frac{n}{x^2+y^2}(-y,x)+t(x,y)\right) \, \Big| \, t \in \R\right\}$ with $\partial P_{T,c}$. The integral $\iint\limits_{\,\,P_{T,c}} \Big|\mathcal{I}_{(x,y)}^n(P_{T,c})\Big|\,dx\,dy$ can be decomposed into a sum of integrals whose integrands are related to these intersection's y-coordinates. Finally, the sum $\sum\limits_{n\neq0} \iint\limits_{\,\,P_{T,c}} \Big|\mathcal{I}_{(x,y)}^n(P_{T,c})\Big|\,dx\,dy$ is calculated by finding a suitable power series for $ \Big|\mathcal{I}_{(x,y)}^n(P_{T,c})\Big|$ and using \cite[Theorem 421]{hardyIntroductionTheoryNumbers2009} to find the asymptotic growth of the sum.
	\subsection{Setting up the Integrals}
	Let $P_{T,c} = \left\{ (x,y)\in\R^2 \, \big| \, |x|y<c , \, 1<y\leq T\right\}$. First, we wish to calculate, for any $n\in \Z\setminus \{0\}$, $T>1$, $c>0$ and $(x,y) \in P_{T,c}$, the value of
	\begin{equation*}
		\Big|\mathcal{I}_{(x,y)}^n(P_{T,c})\Big| = \bigg|\bigg\{ t \in \R \, \bigg| \, \left( \frac{-ny}{x^2+y^2}+tx, \frac{nx}{x^2+y^2}+ty\right) \in P_{T,c} \bigg\} \bigg|
	\end{equation*}
	where $|\cdot|$ is the one-dimensional Lebesgue measure. For any $n\neq0$, $(x,y) \in P_{T,c}$, let 
	\begin{equation*}
		L(n,x,y,t)=\left(\frac{-ny}{x^2+y^2}+tx, \frac{nx}{x^2+y^2}+ty\right), \quad L(n,x,y) = \left\{ L(n,x,y,t)\, \bigg| \, t \in \R\right\}
	\end{equation*}
	Then $L(n,x,y)$ is an infinite line parameterised by $t\in\R$ intersecting the point $\frac{n}{x^2+y^2}\left(-y,x\right)$ with gradient $\frac{y}{x}$ (or a vertical line when $x=0$). Suppose $n>0$. Then
	\begin{equation*}
		\frac{(-n)\pm \sqrt{(-n)^2\pm 4xyc}}{2x} = \frac{n \mp \sqrt{n^2\pm 4xyc}}{-2x} = \frac{n \mp \sqrt{n^2\mp 4(-x)yc}}{2(-x)}
	\end{equation*}
	Furthermore,
	\begin{gather*}
		\left(\frac{-(-n)y}{x^2+y^2} + t(-x), \frac{(-n)(-x)}{x^2+y^2} + ty\right) = \left(-\Big(\frac{-ny}{x^2+y^2} + tx\Big), \frac{nx}{x^2+y^2} + ty\right)
	\end{gather*}
	so, by the symmetry of $P_{T,c}$ around the $y$-axis, $\Big|\mathcal{I}_{(x,y)}^n(P_{T,c})\Big| =\Big|\mathcal{I}_{(x,y)}^{-n}(P_{T,c})\Big|  $. Therefore, we need only examine the case that $n>0$. 
	
	Fix $n\in\N^+$ and $c>0$. Let $Q_{c} = \{(x,y) \in \R^2 \, | \, |xy| \leq c\}$. The line $L$ must intersect $\partial Q_c$ at either 2 or 4 points (the case that $L$ intersects at 2 points and is tangent at a third point is covered by the latter case). The $y$-coordinate of these intersections, for $x\neq 0 $, is given by
	\begin{equation*}
		\frac{n+ (-1)^a \sqrt{n^2 + (-1)^b 4xyc}}{2x}
	\end{equation*}
	where $a, b \in \{0,1\}$. Note that when $x= 0$, there are two intersections with $y$-coordinate $\pm\frac{cy}{n}$. For ease of notation, define
	\begin{gather*}
		y_1 = \frac{n-\sqrt{n^2+4xyc}}{2x} \\
		y_2 = \frac{n-\sqrt{n^2-4xyc}}{2x} \\
		y_3 = \frac{n+\sqrt{n^2-4xyc}}{2x} \\
		y_4 = \frac{n+\sqrt{n^2+4xyc}}{2x} \\
	\end{gather*}
	For $x>0$, and when they exist, these can be ordered as
	\begin{equation*}
		y_1 < 0 < y_2<y_3<y_4
	\end{equation*}
	Let $t_i$ be the value of $t$ such that the y-coordinate of $L(n,x,y,t_i)$ is $y_i$. When $n<2c$, for $t \in [t_1,t_4]$, $L(n,x,y,t)\in Q_c$. When $n\geq2c$, for $t \in [t_1,t_2]\cup [t_3,t_4]$, $L(n,x,y,t) \in Q_c$.
	When $x<0$, 
	\begin{equation*}
		y_3 < y_4 < y_1 < 0 < y_2.
	\end{equation*}
	Similarly, $L(n,x,y,t)\in Q_c$ when $t\in [t_3,t_2]$ and $n<2c$ or when $t\in[t_3,t_4]\cup[t_1,t_2]$ and $n\geq 2c$.
	
	We now wish to find the intersections of $L(n,x,y)$ with $\partial P_{T,c}$. Notice that if $(x,y) \in P_{T,c}$ then $y>1$, so any $y_i<0$ will not correspond to an intersection of $L(n,x,y)$ with $\partial P_{T,c}$. In particular, $y_1<0$ for all $x$, so we need not consider it in these calculations. 
	In view of the above, the line $L(n,x,y)$ can intersect $\partial P_{T,c}$ in the following ways:
	\begin{enumerate}
		\item through the lines $\{y=1\}$ and $\{y=T\}$, 
		\item through the line $\{y=1\}$ and intersecting once more at a point with $y$-coordinate $y_2 \in [1,T]$,
		\item through the line $\{y=1\}$ and intersecting once more at a point with $y$-coordinate $y_4\in [1,T]$,
		\item through the line $\{y=T\}$ and intersecting once more at a point with $y$-coordinate $  y_3  \in [1,T]$,
		\item intersecting twice at points with $y$-coordinates $ y_3 , y_4 \in [1,T]$,
		\item through the line $\{y=1\}$ and intersecting three more times at points with $y$-coordinates $y_2 , y_3 , y_4 \in [1,T]$,
		\item through the lines $\{y=1\}$, $\{y=T\}$ and intersecting twice more at points with $y$-coordinates $ y_2 , y_3\in [1,T]$,
		\item no intersection with $\partial P_{T,c}$.
	\end{enumerate}
	Suppose we have $n,x,y$ such that $L(n,x,y)$ intersects $\partial P_{T,c}$ at two points $(\alpha_1,\beta_1),(\alpha_2,\beta_2)$ with $\beta_2>\beta_1$. Then
	\begin{equation}\label{length}
		\Big|\mathcal{I}_{(x,y)}^n(P_{T,c})\Big|=\left(\frac{\beta_2}{y} - \frac{nx}{yx^2+y^3} \right)- \left(\frac{\beta_1}{y} - \frac{nx}{yx^2+y^3} \right)  = \frac{\beta_2-\beta_1}{y}
	\end{equation}
	Similarly, if $L(n,x,y)$ intersects $\partial P_{T,c}$ at $(\alpha_i,\beta_i)$ for $i=1,2,3,4$ with $\beta_1<\beta_2<\beta_3<\beta_4$, then
	\begin{equation*}
		\big|\mathcal{I}_{(x,y)}^n(P_{T,c})\big|=\frac{\beta_4-\beta_3+\beta_2-\beta_1}{y}
	\end{equation*}
	For each of the 8 statements about intersection above, denote by $B_i$ the set of all $(x,y) \in P_{T,c}$ such that $L(n,x,y)$ intersects $\partial P_{T,c}$ as in statement $i$ above. The collection $\{B_i\}$ is a partition of $P_{T,c}$. The integral we wish to calculate can therefore be written as
	\begin{equation}\label{fulleq}
		\iint\limits_{P_{T,c}}\big| \mathcal{I}_{(x,y)}^n(P_{T,c})\big| = \sum_{i=1}^7 \iint\limits_{B_i} \frac{1}{y}\left( \delta_{T,i} T + \delta_{y_4,i}y_4 - \delta_{y_3,i}y_3 + \delta_{y_2,i}y_2 - \delta_{1,i}\right)  \,dx\,dy
	\end{equation}
	where $\delta_{\alpha,i}=1$ when statement $i$ includes an intersection of type $\alpha\in\{1,y_2,y_3,y_4,T\}$ and is $0$ otherwise. For example, $\delta_{T,1}=\delta_{T,4}=\delta_{T,7}=1$, and $\delta_{\alpha, 8}=0$ for all intersection types.
	To ease the calculations to come, we may use the linearity of the integral to split (\ref{fulleq}) into a sum of five integrals, each with integrand of the form $\frac{\alpha}{y}$. The domains of these integrals are given by $A_\alpha=\bigcup_{i:\delta_{\alpha,i}=1}B_i$. Using the example of $T$ as before, $A_T = B_1\cup B_4 \cup B_7$. Therefore, 
	\begin{align}\label{decomp}
		\iint\limits_{P_{T,c}}\left| \mathcal{I}_{(x,y)}^n(P_{T,c})\right| \,dx\,dy =& \iint\limits_{A_T} \frac{T}{y} \,dx\,dy
		+ \iint\limits_{A_{y_4}}\frac{y_4}{y} \,dx\,dy
		- \iint\limits_{A_{y_3}}\frac{y_3}{y} \,dx\,dy
		\\&+ \iint\limits_{A_{y_2}}\frac{y_2}{y} \,dx\,dy- \iint\limits_{A_1}\frac{1}{y} \,dx\,dy.\nonumber
	\end{align}
	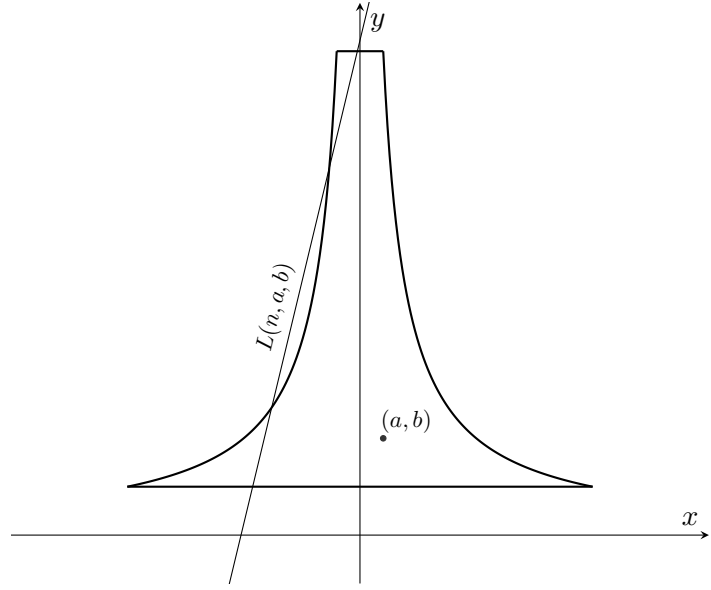
\begin{figure}[h!]
		\centering
		\begin{tikzpicture}
			\begin{axis}[axis y line=middle,axis x line=middle,
				xtick=\empty,ytick=\empty,
				ymin=-1, ymax=11,
				ylabel=$y$, 
				xmin=-1.5,xmax=1.5,xlabel=$x$,
				samples=100,
				width=0.69\textwidth
				]
				\addplot [black,thick,domain=-1:-0.1] {-1/x};
				\addplot [name path = posPT, black,thick,domain=0.1:1] {1/x};
				\addplot [name path = T, black,thick,domain=-0.1:0.1] {10} ;
				\addplot [name path=1,black,thick,domain=-1:1] {1};
				\node[scale=0.8] at (0.2,2.35) {$(a,b)$};
				\node[draw=black!80,fill=black!80,circle,scale=0.2] at (0.1,2) {};
				\addplot [black,thin,domain=-0.6:0.04] {20*x+4100/401} node[black,above,sloped,pos=0.5,scale=0.8] {$L(n,a,b)$};
			\end{axis}
		\end{tikzpicture} 
		\caption{An example of a point $(a,b)\in P_{T,c}$ and a line $L(n,a,b)$. For this line, we have $y_1<1<y_2<y_3<T<y_4$, so $(a,b)\in A_1 \cap A_{y_2} \cap A_{y_3} \cap A_T$.}
	\end{figure}
	Notice also that $A_\alpha$ is the subset of $P_{T,c}$ such that $L(n,x,y)$ has an intersection with $\partial P_{T,c}$ of type $\alpha$. Explicitly, these sets are:
	\begin{enumerate}
		\item[$A_1$:] all $(x,y)$ such that $L(n,x,y)$ intersects the line segment $\left\{\left(x,1\right)\,| \, \left|x\right|<c\right\}\subset \partial P_{T,c}$,
		\item[$A_{y_2}$:] all $(x,y)$ such that $y_2$ exists and $1<y_2<T$,
		\item[$A_{y_3}$:] all $(x,y)$ such that $y_3$ exists and $1<y_3<T$,
		\item[$A_{y_4}$:] all $(x,y)$ such that $y_4$ exists and $1<y_4<T$,
		\item[$A_T$:] all $(x,y)$ such that $L(n,x,y)$ intersects the line segment $\left\{\left(x,T\right)\,|\, \left|x\right|<\frac{c}{T}\right\}\subset \partial P_{T,c}$.
	\end{enumerate}
	Some cases will contain sub-cases; whether or not $n<2c$ affects the boundary conditions for all sets except $A_{y_4}$. To help calculate these areas, we calculate that the derivatives of the $y_i$ with respect to $x$ are
	\begin{gather*}
		\frac{\partial}{\partial x} y_1 =  \frac{\frac{n^2+2xyc}{\sqrt{n^2+4xyc}}-n}{2x^2} \geq 0\\
		\frac{\partial}{\partial x} y_2 =  \frac{\frac{n^2-2xyc}{\sqrt{n^2-4xyc}}-n}{2x^2} \geq 0\\
		\frac{\partial}{\partial x} y_3 = -\frac{\frac{n^2-2xyc}{\sqrt{n^2-4xyc}}+n}{2x^2} \leq 0\\
		\frac{\partial}{\partial x} y_4 = -\frac{\frac{n^2+2xyc}{\sqrt{n^2+4xyc}}+n}{2x^2} \leq 0\\
	\end{gather*}
	\begin{enumerate}
		\item[$A_1$:] Suppose $x>0$. If $x>\frac{n^2}{4yc}$, then $y_2,y_3$ do not exist, therefore the line $L$ intersects $Q_c$ at two points, each point with $y$-coordinate $y_1<0<y_4$. In this case, $L$ intersects the line $\left\{\left(x,1\right)\,| \, \left|x\right|<c\right\}$ when $y_4\geq 1$, so $x<n+cy$. Note that the line $x=n+cy$ lies outside of $P_{T,c}$.
		When $0<x \leq \frac{n^2}{4yc}$, both $y_2$ and $y_3$ exist. In this case, the line $L$ leaves $P_{T,c}$ at $y_2$ and re-enters at $y_3$. So $L$ intersects $\left\{\left(x,1\right)\,| \, \left|x\right|<c\right\}$ when $y_2>1$ or when $y_3<1$, hence $x>n-cy$.
		Finally, when $x<0$, only $y_2>0$, so we require that $y_2>1$, hence $x>n-cy$.
		If $n \leq 2c$, the integral is
		\begin{equation*}
			\iint\limits_{A_1}\frac{1}{y} \,dx\,dy= \int_1^T \int_{\max\{-\frac{c}{y},n-cy\}}^{\frac{c}{y}} \frac{1}{y} \,dx\,dy
		\end{equation*}
		otherwise, 
		\begin{equation*}
			\iint\limits_{A_1}\frac{1}{y} \,dx\,dy= \int_{\frac{n+\sqrt{n^2-4c^2}}{2c}}^T \int_{\max \left\{-\frac{c}{y},n-cy\right\}}^{\frac{c}{y}} \frac{1}{y} \,dx\,dy
		\end{equation*}
		This integral is only valid for $n\leq\frac{c(T^2+1)}{T}$.
		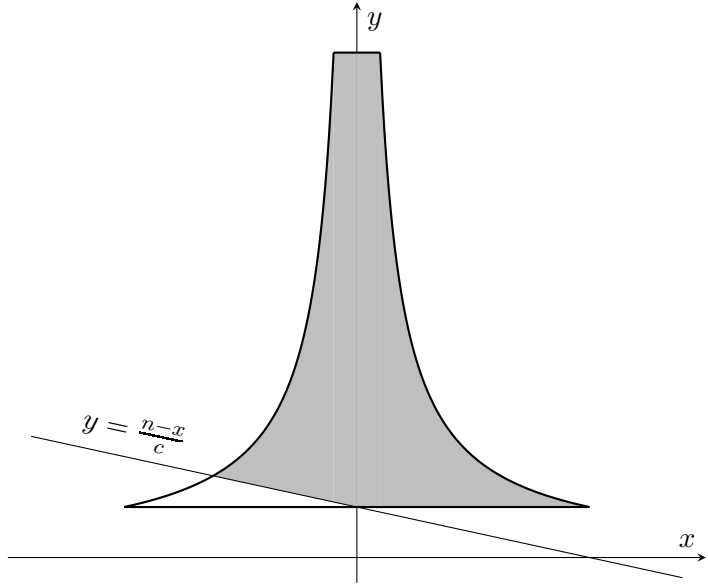
\begin{figure}[H]
			\centering
			\begin{tikzpicture}
				\begin{axis}[axis y line=middle,axis x line=middle,
					xtick=\empty,ytick=\empty,
					ymin=-0.5, ymax=11,
					ylabel=$y$, 
					xmin=-1.5,xmax=1.5,xlabel=$x$,
					samples=100,
					width=0.69\textwidth
					]
					\addplot [name path=minuesPT,black,thick,domain=-1:-0.1] {-1/x};
					\addplot [name path=boundary,black,thick,domain=0.1:1] {1/x};
					\addplot [name path =T,black,thick,domain=-0.1:0.1] {10};
					\addplot [name path=1,black,thick,domain=-1:1] {1};
					\addplot [name path = y2boundary,black,thin, domain =-1.4:1.4] {1-x} node[above,sloped,pos=0.15] {$y=\frac{n-x}{c}$};
					\addplot [lightgray] fill between[
					of = y2boundary and minuesPT, soft clip={domain=-0.618:-0.1}
					];
					\addplot [lightgray] fill between[
					of = y2boundary and T, soft clip={domain=-0.1:0}
					];
					\addplot [lightgray] fill between[
					of = 1 and T, soft clip={domain=0:0.1}
					];
					\addplot [lightgray] fill between[
					of = 1 and boundary, soft clip={domain=0.1:1}
					];
				\end{axis}
			\end{tikzpicture} 
			\caption{The area $A_1\subset P_{T,c}$ such that any $(x,y)\in A_1$ has $L(x,n,y)$ intersecting $\partial P_{T,c}\cap \{y=1\}$. The parameters $n=c=1$ and $T=10$ are used.}
		\end{figure} 
		\item[$A_{y_2}$:] $y_2$ exists when $x<\frac{n^2}{4yc}$. Then
		\begin{gather*}
			\frac{n-\sqrt{n^2-4xyc}}{2x} = 1 \\
			\sqrt{n^2-4xyc} = n-2x \\
			n^2-4xyc = n^2-4xn + 4x^2 \\
			x=n-cy
		\end{gather*}
		Substituting this into the expression for $y_2$ gives
		\begin{equation*}
			\frac{n-\sqrt{n^2-4x\left(\frac{n-x}{c}\right)c}}{2x} = \frac{n-\left|n-2x\right|}{2x} = \begin{cases}
				1 & x\leq \frac{n}{2} \\
				1-\frac{n}{x} & x> \frac{n}{2}
			\end{cases}
		\end{equation*}
		The point with $x$-coordinate $\frac{n}{2}$ on the line $x=n-cy$ has $y$-coordinate $\frac{n}{2c}$, and $\left(\frac{n}{2},\frac{n}{2c}\right) \in P_{T,c}$ if and only if $n=2c$, corresponding to the point $(c,1)\in\partial P_{T,c}$. 
		The line $x=n-cy$ intersects $x=\frac{c}{y}$ at $y = \frac{n\pm \sqrt{n^2-4c^2}}{2c}$, which only exist when $n\geq 2c$.
		\begin{gather*}
			\frac{n-\sqrt{n^2-4xyc}}{2x} = T \\
			\sqrt{n^2-4xyc} = n-2Tx \\
			n^2-4xyc = n^2- 4Txn + 4T^2x^2 \\
			x = \frac{Tn-cy}{T^2} 
		\end{gather*}
		Substituting this into the expression for $y_2$ gives 
		\begin{equation*}
			\frac{n-\sqrt{n^2-4x\left(\frac{Tn-T^2x}{c}\right)}c}{2x} = \frac{n-\left|n-2Tx\right|}{2x} = \begin{cases}
				T & x\leq \frac{n}{2T} \\
				\frac{n}{x} -T &x \geq \frac{n}{2T}
			\end{cases}
		\end{equation*}
		The point with $x$-coordinate $\frac{n}{2T}$ on the line $x=\frac{Tn-cy}{T^2}$ has $y$-coordinate $\frac{nT}{2c}$. Since $\frac{\partial}{\partial x}y_2\geq 0$, and for $n\leq 2c$, this integral is of the form
		\begin{align*}
			\iint\limits_{A_{y_2}} \frac{n-\sqrt{n^2-4xyc}}{2xy}\,dx\,dy =&\int_1^\frac{nT}{2c} \int_{\max\left\{n-cy,-\frac{c}{y}\right\}}^{\frac{n^2}{4yc}}  \frac{n-\sqrt{n^2-4xyc}}{2xy}\,dx\,dy  \\&+ \int_\frac{nT}{2c}^T \int_{-\frac{c}{y}}^{\frac{Tn-cy}{T^2}} \frac{n-\sqrt{n^2-4xyc}}{2xy} \,dx\,dy
		\end{align*}
		If $n>2c$, then $\left(\frac{n}{2T},\frac{nT}{2c}\right) \notin P_{T,c}$, so the integral is of the form
		\begin{equation*}
			\iint\limits_{A_{y_2}} \frac{n-\sqrt{n^2-4xyc}}{2xy}\,dx\,dy =\int_{\frac{n+\sqrt{n^2-4c^2}}{2c}}^T \int_{\max\left\{n-cy,-\frac{c}{y}\right\}}^{\frac{c}{y}}  \frac{n-\sqrt{n^2-4xyc}}{2xy}\,dx\,dy  
		\end{equation*}
		This integral is only valid for $n\leq\frac{c(T^2+1)}{T}$.
		\begin{figure}[H]
			\centering
			\begin{tikzpicture}
				\begin{axis}[axis y line=middle,axis x line=middle,
					xtick=\empty,ytick=\empty,
					ymin=-1, ymax=11,
					ylabel=$y$, 
					xmin=-1.5,xmax=1.5,xlabel=$x$,
					samples=100,
					width=0.69\textwidth
					]
					\addplot [name path=minusPT,black,thick,domain=-1:-0.1] {-1/x};
					\addplot [name path=boundary,black,thick,domain=0.1618:1] {1/x};
					\addplot [black,thick,domain=0.1:0.1618] {1/x};
					\addplot [name path =T,black,thick,domain=-0.1:0.1] {10};
					\addplot [name path=1,black,thick,domain=-1:1] {1};
					\addplot [name path = n24yc,black,thin,domain=0:1.4] {1/(4*x)} node[above,sloped,pos=0.98] {$y=\frac{n^2}{4xc}$};
					\addplot [name path=y3boundary,black,thin,domain= 0:0.115] {10-100*x} node[right,pos=0.91] {$y=\frac{Tn-T^2x}{c}$};
					\addplot [name path = y2boundary,black,thin, domain =-1.4:1.4] {1-x} node[above,sloped,pos=0.15] {$y=\frac{n-x}{c}$};
					\addplot [lightgray] fill between[
					of = y3boundary and 1, soft clip = {domain=0:0.09}
					];
					\addplot [lightgray] fill between[
					of = T and y2boundary, soft clip = {domain = -0.1:0}
					];
					\addplot [lightgray] fill between[
					of = minusPT and y2boundary, soft clip = {domain = -0.618:-0.1}
					];
					\addplot [lightgray] fill between[
					of = n24yc and 1, soft clip = {domain = 0.05:0.25}
					];
				\end{axis}
			\end{tikzpicture} 
			\caption{The area $A_{y_2}\subset P_{T,c}$ such that any $(x,y)\in A_{y_2}$ has $1\leq y_2 \leq T$. The parameters $n=c=1$ and $T=10$ are used.}
		\end{figure}
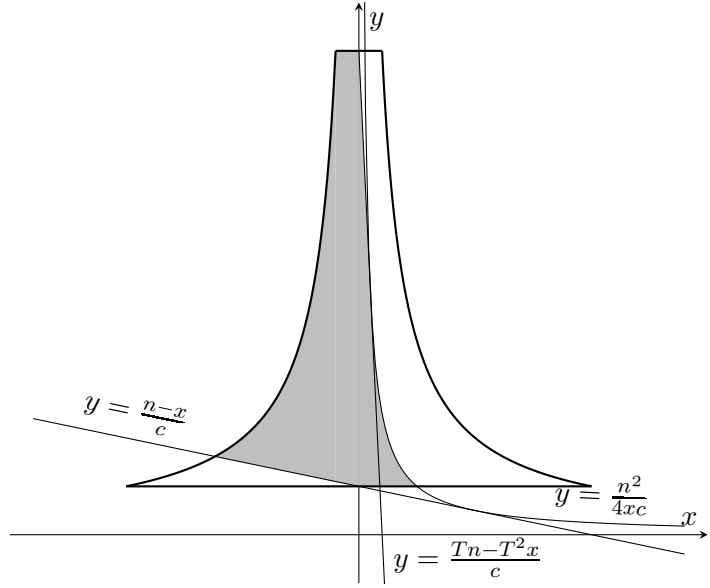 
		\item[$A_{y_3}$:] $y_3$ exists when $x\neq 0$, $x<\frac{n^2}{4yc}$ and $y_3>0$ only when $x>0$. Then 
		\begin{gather*}
			\frac{n+\sqrt{n^2-4xyc}}{2x} = 1 \\
			\sqrt{n^2-4xyc}=2x-n \\
			n^2-4xyc = 4x^2-4xn+n^2 \\
			x = n-cy
		\end{gather*}
		Substituting this into the expression for $y_3$ gives
		\begin{gather*}
			\frac{n+\sqrt{n^2-4x\left(\frac{n-x}{c}\right)c}}{2x} = \frac{n+\left|n-2x\right|}{2x} = \begin{cases}
				\frac{n}{x}-1 & x\leq\frac{n}{2} \\
				1 & x>\frac{n}{2}
			\end{cases}
		\end{gather*}
		On the other hand,
		\begin{gather*}
			\frac{n+\sqrt{n^2-4xyc}}{2x} = T \\
			\sqrt{n^2-4xyc} = 2Tx-n \\
			n^2-4xyc = n^2 - 4Tnx + 4T^2x^2 \\
			x = \frac{Tn-cy}{T^2}
		\end{gather*}
		and substituting this into the expression for $y_3$ gives
		\begin{gather*}
			\frac{n+\sqrt{n^2-4x\left(\frac{Tn-T^2x}{c}\right)c}}{2x} = \frac{n+\left|n-2Tx\right|}{2x} = \begin{cases}
				\frac{n}{x}-T & x \leq \frac{n}{2T}\\
				T & x > \frac{n}{2T}
			\end{cases}
		\end{gather*}
		When $n\leq 2c$, this integral is of the form 
		\begin{equation*}
			\iint\limits_{A_{y_3}}\frac{n+\sqrt{n^2-4xyc}}{2xy}\,dx\,dy=\int_1^{\frac{nT}{2c}} \int_{\frac{Tn-cy}{T^2}}^{\frac{n^2}{4yc}} \frac{n+\sqrt{n^2-4xyc}}{2xy} \,dx\,dy
		\end{equation*}
		\begin{figure}[H]
			\centering
			\begin{tikzpicture}
				\begin{axis}[axis y line=middle,axis x line=middle,
					xtick=\empty,ytick=\empty,
					ymin=-1, ymax=11,
					ylabel=$y$, 
					xmin=-1.5,xmax=1.5,xlabel=$x$,
					samples=100,
					width=0.69\textwidth
					]
					\addplot [black,thick,domain=-1:-0.1] {-1/x};
					\addplot [name path=boundary,black,thick,domain=0.1618:1] {1/x};
					\addplot [black,thick,domain=0.1:0.1618] {1/x};
					\addplot [black,thick,domain=-0.1:0.1] {10};
					\addplot [name path=1,black,thick,domain=-1:1] {1};
					\addplot [name path=Tboundary,black,thin,domain= 0:0.115] {10-100*x} node[right,pos=0.91] {$y=\frac{Tn-T^2x}{c}$};
					\addplot [name path = n24yc,black,thin,domain=0:1.4] {1/(4*x)} node[above,sloped,pos=0.985] {$y=\frac{n^2}{4xc}$};
					\addplot [lightgray] fill between[
					of = Tboundary and n24yc, soft clip = {domain=0.05:9/100}
					];
					\addplot [lightgray] fill between[
					of = n24yc and 1, soft clip = {domain=9/100:0.25}
					];
				\end{axis}
			\end{tikzpicture}
			\caption{The area $A_{y_3}\subset P_{T,c}$ such that any $(x,y)\in A_{y_3}$ has $1\leq y_3 \leq T$. The parameters $n=c=1$ and $T=10$ are used.}
		\end{figure}
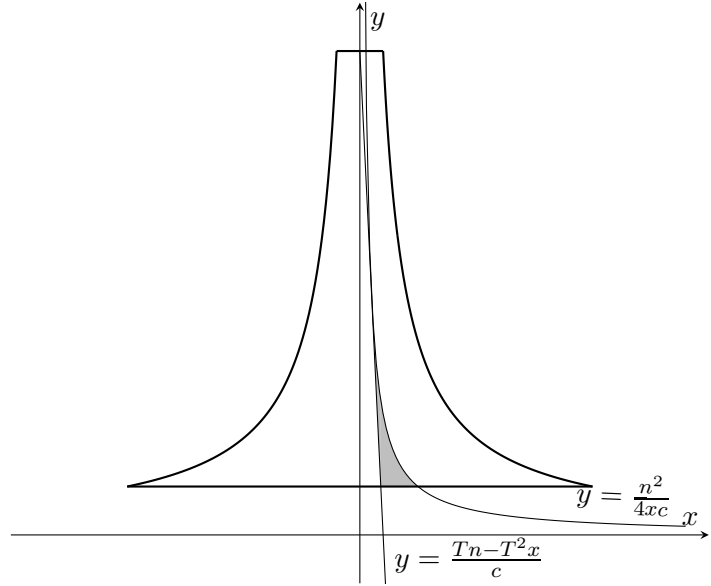  
	\noindent When $n>2c$, the line $x = \frac{Tn-cy}{T^2}$ intersects $x=\frac{c}{y}$ at $y=\frac{T(n-\sqrt{n^2-4c^2})}{2c}$, so the integral becomes
		\begin{equation*}
			\iint\limits_{A_{y_3}}\frac{n+\sqrt{n^2-4xyc}}{2xy}\,dx\,dy= \int_{1}^{\frac{T(n-\sqrt{n^2-4c^2})}{2c}} \int_{\frac{Tn-cy}{T^2}}^{\frac{c}{y}} \frac{n+\sqrt{n^2-4xyc}}{2xy}\,dx\,dy
		\end{equation*}
		This integral is only valid for $n\leq\frac{c\left(T^2+1\right)}{T}$.
		\item[$A_{y_4}$:] $y_4$ is only positive when $x>0$.
		\begin{gather*}
			\frac{n+\sqrt{n^2+4xyc}}{2x} = 1\\
			\sqrt{n^2+4xyc} = 2x - n \\
			n^2+4xyc = 4x^2 - 4xn + n^2 \\
			x = n+cy
		\end{gather*}
		For $c,n>0$, this line does not intersect $P_{T,c}$.
		\begin{gather*}
			\frac{n+\sqrt{n^2+4xyc}}{2x}=T \\
			\sqrt{n^2+4xyc} = 2Tx - n \\
			n^2+4xyc = 4T^2x^2 - 4Txn + n^2 \\
			x= \frac{Tn+yc}{T^2}
		\end{gather*}
		Substituting this into the expression for $y_4$ gives
		\begin{equation*}
			\frac{n+\sqrt{n^2+4x\left(\frac{xT^2-Tn}{c}\right)c}}{2x} = \frac{n+\left|n-2xT\right|}{2x} = \begin{cases}
				T & x\geq \frac{n}{2T}\\
				\frac{n}{x}-T & x < \frac{n}{2T}
			\end{cases}
		\end{equation*}
		The line $x= \frac{Tn+cy}{T^2}$ intersects the line $x=\frac{c}{y}$ at the point $(\frac{n+\sqrt{n^2+4c^2}}{2T}, \frac{T(-n+\sqrt{n^2+4c^2})}{2c})$. Therefore this integral is of the form
		\begin{equation*}
			\iint_{A_{y_4}}\frac{n+\sqrt{n^2+4xyc}}{2xy} \,dx\,dy=\int_1^{\frac{T(-n+\sqrt{n^2+4c^2})}{2c}} \int_{\frac{nT+yc}{T^2}}^{\frac{c}{y}} \frac{n+\sqrt{n^2+4xyc}}{2xy} \,dx\,dy
		\end{equation*}
		This integral is only valid for $n\leq\frac{c(T^2-1)}{T}$.
		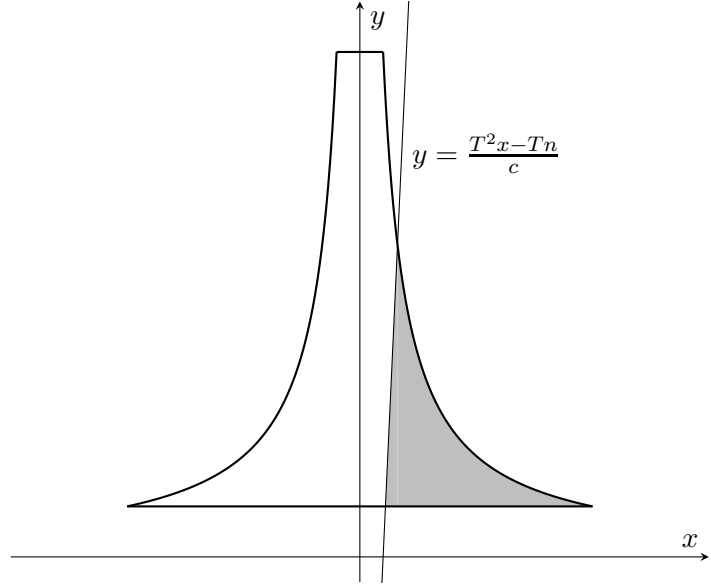
\begin{figure}[H]
			\centering
			\begin{tikzpicture}
				\begin{axis}[axis y line=middle,axis x line=middle,
					xtick=\empty,ytick=\empty,
					ymin=-0.5, ymax=11,
					ylabel=$y$, 
					xmin=-1.5,xmax=1.5,xlabel=$x$,
					samples=100,
					width=0.69\textwidth
					]
					\addplot [black,thick,domain=-1:-0.1] {-1/x};
					\addplot [name path=boundary,black,thick,domain=0.1618:1] {1/x};
					\addplot [black,thick,domain=0.1:0.1618] {1/x};
					\addplot [black,thick,domain=-0.1:0.1] {10};
					\addplot [name path=1,black,thick,domain=-1:1] {1};
					\addplot [name path = y4boundary,black,thin,domain=0:0.3] {100*x-10} node[right,pos=0.6] {$y=\frac{T^2x-Tn}{c}$};
					\addplot [lightgray] fill between[
					of = y4boundary and 1, soft clip = {domain=0.11:0.1618}
					];
					\addplot [lightgray] fill between[
					of = boundary and 1, soft clip = {domain=0.1618:1}
					];
				\end{axis}
			\end{tikzpicture}
			\caption{The area $A_{y_4}\subset P_{T,c}$ such that any $(x,y)\in A_{y_4}$ has $1\leq y_4 \leq T$. The parameters $n=c=1$ and $T=10$ are used.}
		\end{figure} 
		\item[$A_T$:] If $n\leq 2c$, then the integral is 
		\begin{equation*}
			\iint\limits_{A_T} \frac{T}{y}\,dx\,dy = \int_1^T \int_{\frac{Tn-cy}{T^2}}^{\min\{\frac{c}{y},\frac{Tn+cy}{T^2}\}} \frac{T}{y} \,dx\,dy
		\end{equation*}
		otherwise,
		\begin{equation*}
			\iint\limits_{A_T} \frac{T}{y}\,dx\,dy = \int_1^{\frac{T(n-\sqrt{n^2-4c^2})}{2c}} \int_{\frac{Tn-cy}{T^2}}^{\min\{\frac{c}{y},\frac{Tn+cy}{T^2}\}} \frac{T}{y} \,dx\,dy
		\end{equation*}
		This integral is only valid for $n\leq \frac{c(T^2+1)}{T}$.
	\end{enumerate}
	\begin{figure}[H]
		\centering
		\begin{tikzpicture}
			\begin{axis}[axis y line=middle,axis x line=middle,
				xtick=\empty,ytick=\empty,
				ymin=-1, ymax=11,
				ylabel=$y$, 
				xmin=-1.5,xmax=1.5,xlabel=$x$,
				samples=100,
				width=0.69\textwidth
				]
				\addplot [black,thick,domain=-1:-0.1] {-1/x};
				\addplot [name path = posPT, black,thick,domain=0.1:1] {1/x};
				\addplot [name path = T, black,thick,domain=-0.1:0.1] {10} ;
				\addplot [name path = 1, black,thick,domain=-1:1] {1};
				\addplot [name path = y4boundary,black,thin,domain=0:0.3] {100*x-10} node[right,pos=0.6] {$y=\frac{T^2x-Tn}{c}$};
				\addplot [name path = Tboundary,black,thin,domain=-0.07:15/100] {10-100*x} node[right,pos=0.8] {$y=\frac{Tn-T^2x}{c}$};
				\addplot [name path = Tsmall,black,thin,domain=0.09:0.1] {10};
				\addplot [name path = 1small,black,thin,domain=0.09:0.1] {1};
				\addplot [lightgray] fill between[of = Tboundary and T, soft clip = {domain=0:0.0905}];
				\addplot [lightgray] fill between[of = Tsmall and 1small, soft clip = {domain=0.09:0.1}];
				\addplot [lightgray] fill between[of = 1 and posPT, soft clip = {domain=0.1:0.112}];
				\addplot [lightgray] fill between[of = posPT and y4boundary, soft clip = {domain=11/100:0.1618}];
			\end{axis}
		\end{tikzpicture} 
		\caption{The area $A_T\subset P_{T,c}$ such that any $(x,y)\in A_T$ has $L(n,x,y)$ intersecting $P_{T,c}\cap \{y=T\}$. The parameters $n=c=1$ and $T=10$ are used.}
	\end{figure}
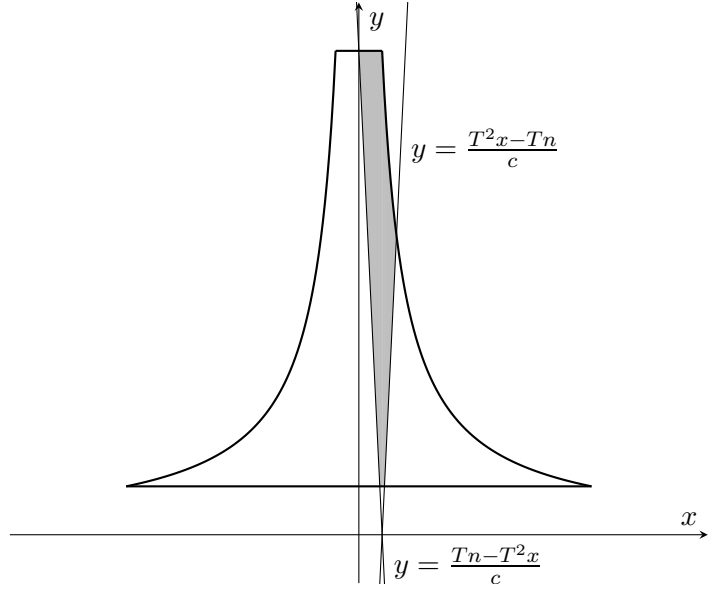
	\subsection{Calculating the Integrals} The following calculations have been verified using Wolfram Mathematica.
	\begin{enumerate}
		\item[$A_1$]	\begin{enumerate}
			\item[] Case $n\leq 2c$
			\begin{gather*}
				\int_1^T \int_{\max \left\{-\frac{c}{y},n-cy\right\}}^{\frac{c}{y}} \frac{1}{y} \,dx\,dy \\
				=- \frac{2c}{T} + \sqrt{n^2+4c^2}+n \log \left(\frac{\sqrt{n^2+4c^2}-n}{2c}\right)\\
				=O\left(1\right)
			\end{gather*} 
			\item[] Case $n>2c$
			\begin{gather*}
				\int_{\frac{n+\sqrt{n^2-4c^2}}{2c}}^T \int_{\max \left\{-\frac{c}{y},n-cy\right\}}^{\frac{c}{y}} \frac{1}{y} \,dx\,dy \\
				= -\frac{2c}{T} + \sqrt{n^2+4c^2}-  \sqrt{n^2-4c^2} + n\log \left( \frac{\sqrt{n^2+4c^2}-n}{n-\sqrt{n^2-4c^2}} \right)
			\end{gather*}
		\end{enumerate}
		\item[$A_{y_2}$] \begin{enumerate}
		\item[] Case $n\leq 2c$
		\begin{gather*}
			\int_1^\frac{nT}{2c} \int_{\max\left\{n-cy,-\frac{c}{y}\right\}}^{\frac{n^2}{4yc}}  \frac{n-\sqrt{n^2-4xyc}}{2xy}\,dx\,dy +\int_{\frac{nT}{2c}}^T \int_{-\frac{c}{y}}^{\frac{Tn-cy}{T^2}}  \frac{n-\sqrt{n^2-4xyc}}{2xy}\,dx\,dy\\
			=  \sqrt{n^2+4c^2}-\sqrt{n^2-4c^2}+\frac{n}{2}\log\left(n+\sqrt{n^2-4c^2}\right)^2 - \frac{n}{2}\log\left(n+\sqrt{n^2+4c^2}\right)^2 \\- \left(n+\sqrt{n^2-4c^2}-n\log\left(n+\sqrt{n^2-4c^2}\right)\right)\log\left(\frac{n+\sqrt{n^2+4c^2}}{n+\sqrt{n^2-4c^2}}\right) \\+ \left(\sqrt{n^2+4c^2}-\sqrt{n^2-4c^2}+n\log\left(\frac{n+\sqrt{n^2-4c^2}}{n+\sqrt{n^2+4c^2}}\right)\right)\log\left(\frac{T\left(\sqrt{n^2+4c^2}-n\right)}{2c}\right) \\
			=O\left(\log T\right)
		\end{gather*}
		\item[] Case $n>2c$
		\begin{gather*}
			\int_{\frac{n+\sqrt{n^2-4c^2}}{2c}}^T \int_{\max\left\{n-cy,-\frac{c}{y}\right\}}^{\frac{c}{y}} \frac{n-\sqrt{n^2-4xyc}}{2xy}   \,dx\,dy\\
			= \left(n\log\left(n+\sqrt{n^2-4c^2}\right) -\sqrt{n^2-4c^2}-n\right) \log \left(\frac{n+\sqrt{n^2+4c^2}}{n+\sqrt{n^2-4c^2}}\right) \\- \frac{n}{2}\log^2\left( n+\sqrt{n^2+4c^2}\right)  +\frac{n}{2}\log^2\left( n+\sqrt{n^2-4c^2}\right) + \left(\sqrt{n^2+4c^2}-\sqrt{n^2-4c^2}\right) \\+\left( n\log\left(\frac{n+\sqrt{n^2-4c^2}}{n+\sqrt{n^2+4c^2}}\right) +\sqrt{n^2+4c^2}- \sqrt{n^2-4c^2}\right)\left(\log T - \log\left(\frac{n+\sqrt{n^2+4c^2}}{2c}\right) \right) 
		\end{gather*}
	\end{enumerate}
		\item[$A_{y_3}$] 
		\begin{enumerate}
			\item[] Case $n\leq 2c$
			\begin{align*}
				&\int_{1}^{\frac{nT}{2c}} \int_{\frac{Tn-cy}{T^2}}^{\frac{n^2}{4yc}} \frac{n+\sqrt{n^2-4xyc}}{2xy}\,dx\,dy \\
				=&\frac{n}{2} \log^2\left(\frac{nT}{2c}\right) + O\left(\log T\right)
			\end{align*}
			\item[] Case $n>2c$
			\begin{align*}
				&\int_{1}^{\frac{T\left(n-\sqrt{n^2-4c^2}\right)}{2c}} \int_{\frac{Tn-cy}{T^2}}^{\frac{c}{y}} \frac{n+\sqrt{n^2-4xyc}}{2xy}\,dx\,dy \\
				=& \left(n \log \left(n-\sqrt{n^2-4c^2}\right) + \sqrt{n^2-4c^2} - n \right)\log \left(\frac{T(n-\sqrt{n^2-4c^2})}{2c}\right) \\&- \frac{n}{2}\log^2\left(n-\sqrt{n^2-4c^2}\right) +n -\sqrt{n^2-4c^2} + \frac{n}{2}\log^2\left( \frac{2c}{T}\right) - \frac{2c}{T} \\
			\end{align*}
		\end{enumerate}
	\item[$A_{y_4}$] 
	\begin{align*}
		&\int_{1}^{\frac{T(\sqrt{n^2+4c^2}-n)}{2c}} \int_{\frac{nT+yc}{T^2}}^{\frac{c}{y}} \frac{n+\sqrt{n^2+4xyc}}{2xy}\,dx\,dy \\
		=& \left(n\log\left(\sqrt{n^2+4c^2}-n\right) +\sqrt{n^2+4c^2}-n \right) \log \left(\frac{T(\sqrt{n^2+4c^2}-n)}{2c}\right) \\&+n - \sqrt{n^2+4c^2} - \frac{1}{2}n\log^2\left(\sqrt{n^2+4c^2}-n\right) +\frac{2c}{T}+ \frac{n}{2}\log^2\left(\frac{2c}{T}\right) \\
		=& \frac{n}{2}\log^2\left(\frac{2c}{T}\right) + O\left(\log T\right)
	\end{align*}
		\item[$A_T$] \begin{enumerate}
			\item[] Case $n\leq 2c$ 
			\begin{align*}
				&\int_1^T \int_{\frac{Tn-cy}{T^2}}^{\min\left\{\frac{c}{y},\frac{Tn+cy}{T^2}\right\}} \frac{T}{y}  \,dx\,dy \\
				=& \sqrt{n^2+4c^2} - \frac{2c}{T} + n\log\left(\frac{\sqrt{n^2+4c^2}-n}{2c}\right) \\
				=& O(1) 
			\end{align*}
			\item[] Case $n>2c$
			\begin{align*}
				&\int_1^{\frac{T\left(n-\sqrt{n^2-4c^2}\right)}{2c}} \int_{\frac{Tn-cy}{T^2}}^{\min\left\{\frac{c}{y},\frac{Tn+cy}{T^2}\right\}} \frac{T}{y}  \,dx\,dy \\
				=& \sqrt{n^2+4c^2}-\sqrt{n^2-4c^2} - \frac{2c}{T} + n\log \left(\frac{\sqrt{n^2+4c^2}-n}{n-\sqrt{n^2-4c^2}}\right)
			\end{align*}
		\end{enumerate}
	\end{enumerate}
	We now calculate $\iint\limits_{P_{T,c}}\left| \mathcal{I}_{(x,y)}^n(P_{T,c})\right| \,dx\,dy$ using equation (\ref{decomp}) and the calculations of the previous section.
	
	When $n\leq2c$ we can bound this integral by
		\begin{gather*}
			O\left( 1\right) + \frac{n}{2}\log^2\left(\frac{2c}{T}\right) + O\left(\log T\right) - \frac{n}{2} \log^2\left(\frac{nT}{2c}\right) + O\left(\log T\right) + O\left(\log T\right) - O(1) \\
			= O\left(\log T \right)
		\end{gather*}
	When $n>2c$, the summand in this case is equal to
		\begin{gather*}
			\log T \Bigg( n \log \left(\frac{\left(\sqrt{n^2+4c^2}-n\right)\left(n+\sqrt{n^2-4c^2}\right)}{\left(n+\sqrt{n^2+4c^2}\right)\left(n-\sqrt{n^2-4c^2}\right)}\right)+ 2\left( \sqrt{n^2+4c^2}-\sqrt{n^2-4c^2}\right) \Bigg) \\
			+ n\left(1+\log \left(2c\right)\right) \log \left( \frac{\left(n+\sqrt{n^2-4c^2}\right)\left(n-\sqrt{n^2-4c^2}\right)}{\left(\sqrt{n^2+4c^2}-n\right)\left(n+\sqrt{n^2+4c^2}\right)} \right) \\
			+ \sqrt{n^2+4c^2} \log \left( \frac{\sqrt{n^2+4c^2}-n}{\sqrt{n^2+4c^2}+n}\right) - \sqrt{n^2-4c^2} \log \left(\frac{n-\sqrt{n^2-4c^2}}{n+\sqrt{n^2-4c^2}}\right)+ \frac{4c}{T}  \\
			+ \frac{n}{2} \Bigg(\log^2\left(\sqrt{n^2+4c^2}-n\right)+\log^2\left(\sqrt{n^2+4c^2}+n\right)\\-\log^2\left(n+\sqrt{n^2-4c^2}\right)-\log^2\left(n-\sqrt{n^2-4c^2}\right) \Bigg)
		\end{gather*}
		Notice that
		\begin{equation*}
			\frac{\left(n+\sqrt{n^2-4c^2}\right)\left(n-\sqrt{n^2-4c^2}\right)}{\left(\sqrt{n^2+4c^2}-n\right)\left(n+\sqrt{n^2+4c^2}\right)} = 1
		\end{equation*}
		so the summand is equal to
		\begin{gather*}
			\log T \Bigg( n \log \left(\frac{\left(\sqrt{n^2+4c^2}-n\right)\left(n+\sqrt{n^2-4c^2}\right)}{\left(n+\sqrt{n^2+4c^2}\right)\left(n-\sqrt{n^2-4c^2}\right)}\right)+ 2\left( \sqrt{n^2+4c^2}-\sqrt{n^2-4c^2}\right) \Bigg) \\
			+ \sqrt{n^2+4c^2} \log \left( \frac{\sqrt{n^2+4c^2}-n}{\sqrt{n^2+4c^2}+n}\right) - \sqrt{n^2-4c^2} \log \left(\frac{n-\sqrt{n^2-4c^2}}{n+\sqrt{n^2-4c^2}}\right) + \frac{4c}{T}\\
			+ \frac{n}{2} \Bigg(\log^2\left(\sqrt{n^2+4c^2}-n\right)+\log^2\left(\sqrt{n^2+4c^2}+n\right)\\-\log^2\left(n+\sqrt{n^2-4c^2}\right)-\log^2\left(n-\sqrt{n^2-4c^2}\right) \Bigg)
		\end{gather*}
	\subsection{Calculating the Second Moment}
	Recall that
	\begin{equation*}
		\norm{\widetilde{\mathbbm{1}}_{P_{T,c}}}_2^2 = \frac{1}{\zeta(2)}\left(\textnormal{area}(P_{T,c}) + \sum_{n\neq 0} \frac{\varphi(\left|n\right|)}{|\left|n\right|} \iint\limits_{P_{T,c}}\left| \mathcal{I}_{(x,y)}^n(P_{T,c})\right| \,dx\,dy\right)
	\end{equation*}
	Since $|\mathcal{I}_{(x,y)}^n(P_{T,c})| =|\mathcal{I}_{(x,y)}^{-n}(P_{T,c})|$, we can rewrite the sum as
	\begin{equation*}\label{simp}
		\sum_{n\neq 0} \frac{\varphi(\left|n\right|)}{|\left|n\right|} \iint\limits_{P_{T,c}}\left| \mathcal{I}_{(x,y)}^n(P_{T,c})\right| \,dx\,dy =2 \sum_{n\in \N\setminus\{0\}} \frac{\varphi(\left|n\right|)}{\left|n\right|} \iint\limits_{P_{T,c}}\left| \mathcal{I}_{(x,y)}^n(P_{T,c})\right| \,dx\,dy
	\end{equation*}
	Therefore,
	\begin{equation}\label{simplified2norm}
		\norm{\widetilde{\mathbbm{1}}_{P_{T,c}}}_2^2 = \frac{1}{\zeta(2)}\textnormal{area}(P_{T,c}) + \frac{2}{\zeta(2)}\sum_{n\in \N\setminus\{0\}} \frac{\varphi(\left|n\right|)}{\left|n\right|} \iint\limits_{P_{T,c}}\left| \mathcal{I}_{(x,y)}^n(P_{T,c})\right| \,dx\,dy
	\end{equation}
	We split the sum in (\ref{simplified2norm}) into the cases $n\leq 2c$ and $n>2c$.
	\begin{gather*}
		\sum_{n\in \N\setminus\{0\}} \frac{\varphi(\left|n\right|)}{\left|n\right|} \iint\limits_{P_{T,c}}\left| \mathcal{I}_{(x,y)}^n(P_{T,c})\right| \,dx\,dy \\= \sum_{0<n\leq2c: n \in \N} \frac{\varphi(\left|n\right|)}{\left|n\right|} \iint\limits_{P_{T,c}}\left| \mathcal{I}_{(x,y)}^n(P_{T,c})\right| \,dx\,dy + \sum_{2c<n: n \in \N} \frac{\varphi(\left|n\right|)}{\left|n\right|} \iint\limits_{P_{T,c}}\left| \mathcal{I}_{(x,y)}^n(P_{T,c})\right| \,dx\,dy
	\end{gather*}
	Then the first of these sums can be bounded as
	\begin{gather*}
		\sum_{0<n\leq2c: n \in \N} \frac{\varphi(\left|n\right|)}{\left|n\right|} \iint\limits_{P_{T,c}}\left| \mathcal{I}_{(x,y)}^n(P_{T,c})\right| \,dx\,dy \\
		\leq C+\sum_{n=1}^{\lceil 2c \rceil} \frac{\varphi(n)}{n} O_c\left(\log T\right) \\
		= O_c (\log T)
	\end{gather*}
	The second sum is as follows.
	\begin{gather*}
		\sum_{n=\lceil 2c \rceil }^{\lfloor cT + \frac{c}{T}\rfloor} \frac{\varphi(n)}{|n|} \Bigg[\log T  \Bigg(  n \log \left(\frac{\left(\sqrt{n^2+4c^2}-n\right)\left(n+\sqrt{n^2-4c^2}\right)}{\left(n+\sqrt{n^2+4c^2}\right)\left(n-\sqrt{n^2-4c^2}\right)}\right)\\ + \frac{4c}{T}+ 2\left( \sqrt{n^2+4c^2}-\sqrt{n^2-4c^2}\right) \Bigg) + \sqrt{n^2+4c^2} \log \left( \frac{\sqrt{n^2+4c^2}-n}{n+\sqrt{n^2+4c^2}}\right) \\
		+ \sqrt{n^2-4c^2} \log \left(\frac{n+\sqrt{n^2-4c^2}}{n-\sqrt{n^2-4c^2}}\right)+ \frac{n}{2} \Bigg(\log^2\left(\sqrt{n^2+4c^2}-n\right)+\log^2\left(\sqrt{n^2+4c^2}+n\right)\\-\log^2\left(n+\sqrt{n^2-4c^2}\right)-\log^2\left(n-\sqrt{n^2-4c^2}\right) \Bigg)\Bigg] - \epsilon_{y_4}
	\end{gather*}
	Here, $\epsilon_{y_4}$ is equal to $\iint\limits_{A_{y_4}} \frac{y_4}{y} \, dx\, dy$ evaluated at $n = \lfloor cT+\frac{c}{T}\rfloor$ and appears in the expression since this particular integral does not validly contribute to the value of $\norm{\widetilde{\mathbbm{1}}_{P_{T,c}}}_2^2$ for this value of $n$ (as shown above, it is only valid for $n\leq\frac{c(T^2-1)}{T}$), hence we deduct its value from this sum. We may calculate, for $n=\lfloor cT+\frac{c}{T}\rfloor$, that
	\begin{equation*}
		\lim_{T\rightarrow\infty} \iint\limits_{A_{y_4}} \frac{y_4}{y} \, dx\, dy = 0
	\end{equation*}
	hence $\epsilon_{y_4}\rightarrow 0$. Since $\epsilon_{y_4}$ does not diverge as $T\rightarrow\infty$, we will ignore it when calculating the growth of the second sum.
	\subsection{Power Series for the Summands}
	We now wish to find suitable power series for each of the summands of the second sum which will allow us to calculate its asymptotic growth. Using 
	\begin{equation*}
		\sqrt{z+1} = \sum_{k=0}^\infty \begin{pmatrix} \frac{1}{2}\\ k\end{pmatrix} \frac{1}{z^{k-\frac{1}{2}}} \, , \quad x>1
	\end{equation*}
	\begin{equation*}
		\sqrt{z-1} = \sum_{k=0}^\infty (-1)^k\begin{pmatrix} \frac{1}{2}\\ k\end{pmatrix} \frac{1}{z^{k-\frac{1}{2}}} \, , \quad x>1
	\end{equation*}
	we calculate:
	\begin{equation*}
		\sqrt{n^2-4c^2} =  2c\cdot \sqrt{\frac{n^2}{4c^2}-1}  = \sum_{k=0}^\infty (-1)^k\begin{pmatrix}
			\frac{1}{2} \\ k
		\end{pmatrix} \frac{4^kc^{2k}}{n^{2k-1}}\, , \quad n>2c
	\end{equation*}
	\begin{equation*}
		\sqrt{n^2+4c^2} = 2c\cdot \sqrt{\frac{n^2}{4c^2}+1} = \sum_{k=0}^\infty\begin{pmatrix}
			\frac{1}{2} \\ k
		\end{pmatrix}  \frac{4^kc^{2k}}{n^{2k-1}} \, , \quad n>2c
	\end{equation*}
	To find a power series for the $\log\left(\sqrt{n^2+4c^2}-n\right)$, consider
	\begin{equation*}
		\frac{\partial }{\partial n} \left( \log \left(\sqrt{n^2+4c^2}-n\right)- \log\left( \frac{2c^2}{n}\right) \right)= \frac{1}{n} - \frac{1}{\sqrt{n^2+4c^2}}
	\end{equation*}
	Recall that
	\begin{equation}\label{Laurent1}
		\frac{1}{\sqrt{1-x^2}}= \sum_{k=0}^\infty \begin{pmatrix} 2k \\k\end{pmatrix} \frac{1}{4^k} \frac{1}{x^{2k+1}}\, , \quad x>1
	\end{equation}
	\begin{equation}\label{Laurent2}
		\frac{1}{\sqrt{1+x^2}}= \sum_{k=0}^\infty \begin{pmatrix} 2k \\k\end{pmatrix} \frac{(-1)^k}{4^k} \frac{1}{x^{2k+1}}\, , \quad x>1
	\end{equation}
	To calculate the power series for $n\log\left(\sqrt{n^2+4c^2}-n\right)$, notice that
	\begin{equation*}
		\frac{\partial }{\partial n} \log\left( \frac{\sqrt{n^4+4c^2n^2}-n^2}{2c^2} \right) = \frac{1}{n}-\frac{1}{\sqrt{n^2+4c^2}}
	\end{equation*}
	By equation (\ref{Laurent2}), 
	\begin{equation*}
		\frac{1}{n}-\frac{1}{\sqrt{n^2+4c^2}} = \frac{2c^2}{n^3}-\frac{6c^4}{n^5}+\frac{20c^6}{n^7} - \frac{70c^8}{n^9} \cdots
	\end{equation*}
	Integrating with respect to $n$ yields
	\begin{equation*}
		\log\left( \frac{\sqrt{n^4+4c^2n^2}-n^2}{2c^2} \right) = \lim_{n\rightarrow\infty}\log\left( \frac{\sqrt{n^4+4c^2n^2}-n^2}{2c^2} \right) - \frac{c^2}{n^2}+\frac{3c^4}{2n^4} - \frac{10c^6}{3n^6} + \frac{35c^8}{4n^8}-\cdots
	\end{equation*}
	Since $\lim_{n\rightarrow\infty}\log\left( \frac{\sqrt{n^4+4c^2n^2}-n^2}{2c^2} \right)=0$, rearranging this equation and multiplying by $n$ gives the power series
	\begin{equation*}
		n \log\left(\sqrt{n^2+4c^2}-n\right) = n\log\left(\frac{2c^2}{n}\right)- \frac{c^2}{n}+\frac{3c^4}{2n^3} - \frac{10c^6}{3n^5} + \frac{35c^8}{4n^7}- \cdots
	\end{equation*}
	Similarly, using (\ref{Laurent1}) and (\ref{Laurent2}),
	\begin{equation*}
		n\log\left(n+\sqrt{n^2-4c^2} \right) = \log\left(2n\right) - \frac{c^2}{n}-\frac{3c^4}{2n^3}-\frac{10c^6}{3n^5}-\frac{35c^8}{4n^7}+\cdots \, , \quad n>2c
	\end{equation*}
	\begin{equation*}
		n\log\left(n+\sqrt{n^2+4c^2}\right) = \log\left(2n\right) + \frac{c^2}{n}-\frac{3c^4}{2n^3}+\frac{10c^6}{3n^5}-\frac{35c^8}{4n^7}+\cdots \, , \quad n>2c
	\end{equation*}
	\begin{equation*}
		n\log\left( n-\sqrt{n^2-4c^2}\right) = \log\left(\frac{2c^2}{n}\right) +\frac{c^2}{n}+\frac{3c^4}{2n^3}+\frac{10c^6}{3n^5}+\frac{35c^8}{4n^7}+\cdots \, , \quad n>2c
	\end{equation*}
	Using these equations, we find the following power series:
	\begin{equation*}
		n\log\left(\frac{(\sqrt{n^2+4c^2}-n)(n+\sqrt{n^2-4c^2})}{(n+\sqrt{n^2+4c^2})(n-\sqrt{n^2-4c^2})}\right)= -\frac{4c^2}{n}-\frac{40c^6}{3n^5}+\cdots
	\end{equation*}
	\begin{equation*}
		\sqrt{n^2+4c^2}-\sqrt{n^2-4c^2}= \frac{4c^2}{n}+\frac{8c^6}{n^5}+\cdots
	\end{equation*}
	\begin{equation*}
		\sqrt{n^2+4c^2} \log \left( \frac{\sqrt{n^2+4c^2}-n}{n+\sqrt{n^2+4c^2}}\right) = 2n\left(\log c - \log n \right) + \frac{4c^2\left(\log c -\log n - \frac{1}{2}\right)}{n}+\cdots
	\end{equation*}
	\begin{equation*}
		\sqrt{n^2-4c^2} \log \left(\frac{n+\sqrt{n^2-4c^2}}{n-\sqrt{n^2-4c^2}}\right) = -2n\left(\log c - \log n \right)+ \frac{4c^2\left(\log c - \log n - \frac{1}{2}\right)}{n}+\cdots
	\end{equation*}
	\begin{gather*}
		\frac{n}{2}\Bigg( \log^2\left(\sqrt{n^2+4c^2}-n\right)+\log^2\left(\sqrt{n^2+4c^2}+n\right)\\-\log^2\left(n+\sqrt{n^2-4c^2}\right)-\log^2\left(n-\sqrt{n^2-4c^2}\right) \Bigg) = \frac{4c^2(\log n -\log c)}{n} + \frac{80c^6\log n}{3n^6} + \cdots
	\end{gather*}	
	Substituting these into the sum gives the following bound on the second sum.
	\begin{gather*}
		\sum_{n=\lceil 2c \rceil }^{\lfloor cT + \frac{c}{T}\rfloor} \frac{\varphi(n)}{|n|} \Bigg[\log T   \Bigg\{  \left(-\frac{4c^2}{n}-\frac{40c^6}{3n^5} + \cdots\right) + 
		2\left(\frac{4c^2}{n}+\frac{8c^6}{n^5} + \cdots \right) \Bigg\} 
		\\+  \left(2n\left(\log c - \log n \right) + \frac{4c^2\left(\log c -\log n - \frac{1}{2}\right)}{n}+\cdots\right)\\
		+\left(-2n\left(\log c - \log n \right)+ \frac{4c^2\left(\log c - \log n - \frac{1}{2}\right)}{n}+\cdots\right) \\+\frac{4c}{T} 
		+ \left(\frac{4c^2(\log n -\log c)}{n} + \frac{80c^6\log n}{3n^6} + \cdots\right)   \Bigg]-\epsilon_{y_4}
	\end{gather*}
	\begin{gather*}
		=\sum_{n=\lceil 2c \rceil }^{\lfloor cT + \frac{c}{T}\rfloor} \frac{\varphi(n)}{n} \Bigg[ \frac{4c^2\log T}{n} - \frac{4c^2\log n }{n} + \frac{4c^2(\log c -1)}{n} + \cdots \Bigg] -\epsilon_{y_4}\\
		= \sum_{n=\lceil 2c \rceil }^{\lfloor cT + \frac{c}{T}\rfloor} \frac{\varphi(n)}{n} \frac{4c^2}{n} \log \left(\frac{T}{n}\right) + O\left(\log T\right)
	\end{gather*}
	We will make use of the following theorem.
	\begin{theorem}[{{\cite[Theorem 421]{hardyIntroductionTheoryNumbers2009}}}]\label{sumint}
		Let $\{c_i\}_{i=1}^\infty$ be a sequence of real numbers and $f:\R\rightarrow \R$. Suppose there exists $s\in \R$ such that $c_j=0$ whenever $j< s$ and $f(t)$ has continuous derivative at all $t\geq s$. Then
		\begin{equation*}
			\sum_{n\leq x} c_nf(n) = \left( \sum_{n\leq x} c_n \right) f(x) - \int_s^x \left( \sum_{n\leq t} c_n \right) \cdot f'(t) dt
		\end{equation*}
	\end{theorem}
	We also have the following:
	\begin{theorem}[{{\cite[Chapter 3]{Walfisz}}}]\label{walf}
		\begin{equation*}
			\sum_{n=1}^N \frac{\varphi(n)}{n} = \frac{N}{\zeta(2)} + O\left(\left(\log N \right)^\frac{2}{3}\left(\log \log N\right)^\frac{4}{3}\right)
		\end{equation*}
	\end{theorem}
	As a consequence of theorem \ref{walf},
	\begin{equation*}
		\sum_{n =\lceil 2c \rceil}^{\lfloor t \rfloor} \frac{\varphi(n)}{n} = \frac{\lfloor t\rfloor- \lceil 2c \rceil +1}{\zeta(2)} + O\left(\left(\log t \right)^\frac{2}{3}\left(\log \log t\right)^\frac{4}{3}\right)
	\end{equation*}
	Take $c_n = \frac{\varphi(n)}{n}$ for $n \geq \lceil 2c \rceil$, $c_n=0$ otherwise, and $f(n) = \frac{1}{n}\log \left( \frac{T}{n}\right)$. These satisfy the assumptions of theorem \ref{sumint}, therefore for large enough $T$,
	\begin{align}
		&\sum_{n=\lceil 2c \rceil}^{\lfloor cT+\frac{c}{T} \rfloor} \frac{\varphi(n)}{n} \frac{1}{n} \log \left(\frac{T}{n}\right)\nonumber\\ 
		=&  \left[ \frac{1}{\zeta(2)} \left( \lfloor cT+\frac{c}{T} \rfloor - \lceil 2c \rceil \right) +O\left( (\log T)^{\frac{2}{3}}(\log \log T)^\frac{4}{3}\right)\right] \frac{1}{\lfloor cT+\frac{c}{T} \rfloor}\log \left(\frac{T}{\lfloor cT+\frac{c}{T} \rfloor}\right) \nonumber\\
		&+ \int_{\lceil 2c \rceil}^{\lfloor cT+\frac{c}{T} \rfloor}\left( \frac{\lfloor t \rfloor -\lceil 2c +1 \rceil }{\zeta(2)} + O\left(\left(\log t \right)^\frac{2}{3}\left(\log \log t\right)^\frac{4}{3} \right) \right)\frac{\log \left( \frac{T}{t} \right)+1}{t^2} dt \nonumber \\
		=&\left( \frac{1}{\zeta(2)} - \frac{\lceil 2c \rceil}{\zeta(2)\lfloor cT+\frac{c}{T}\rfloor}+ \frac{O\left(\left(\log T\right)^\frac{2}{3}\left(\log \log T\right)^\frac{4}{3}\right)}{\lfloor cT + \frac{c}{T}\rfloor} \right) \log \left( \frac{T}{\lfloor cT+\frac{c}{T}\rfloor}\right) \nonumber \\
		&+ \frac{1}{\zeta(2)} \int_{\lceil 2c \rceil}^{\lfloor cT+\frac{c}{T} \rfloor} \frac{\lfloor t \rfloor}{t^2} \left( \log\left(\frac{T}{t}\right)+1\right)  \,dt- \frac{\lceil 2c+1\rceil}{\zeta(2)} \int_{\lceil 2c \rceil}^{\lfloor cT+\frac{c}{T} \rfloor} \frac{1}{t^2}\left( \log\left(\frac{T}{t}\right)+1\right) \,dt \nonumber \\
		&+ \frac{1}{\zeta(2)} \int_{\lceil 2c \rceil}^{\lfloor cT+\frac{c}{T} \rfloor} \frac{O\left(\left(\log t\right)^\frac{2}{3}\left(\log \log t\right)^\frac{4}{3}\right)}{t^2}\left( \log\left(\frac{T}{t}\right)+1\right)  \,dt \nonumber \\
		:=& \frac{1}{\zeta(2)} \int_{\lceil 2c \rceil}^{\lfloor cT+\frac{c}{T}\rfloor} \frac{\lfloor t \rfloor}{t^2} \log\left(\frac{T}{t}\right) dt + E'(T) \nonumber \\
		:=&  \frac{1}{\zeta(2)}\int_{1}^T \frac{\log \left(\frac{T}{t}\right)}{t} dt +E''(T)+E'(T)\nonumber\\
		=& \frac{1}{2\zeta(2)}  \log^2(T) +E''(T)+E'(T) \label{finalcalc}
	\end{align}
	Let $E(T):=E''(T)+E'(T)$. Then equations (\ref{simplified2norm}) and (\ref{finalcalc}) imply
	\begin{align*}
		&\int_{X^2} \left(\widetilde{\mathbbm{1}}_{P_{T,c}} - \frac{\textnormal{area}(P_{T,c})}{\zeta(2)}\right)^2\,d\mu \\=& \norm{\widetilde{\mathbbm{1}}_{P_{T,c}}}_2^2 - \frac{4c^2}{\zeta(2)} \log^2(T) \\
		=& \frac{1}{\zeta(2)}\left(\text{area}(P_{T,c}) + 2\sum_{n\in \N}\frac{\varphi(|n|)}{|n|}\iint\limits_{P_{T,c}}\bigg|\mathcal{I}_{(x,y)}^n(P_{T,c}) \bigg|\,dx\,dy\right) - \frac{4c^2}{\zeta(2)}\log^2(T) \\
		=& O\left(\log T\right)+\frac{2}{\zeta(2)}\frac{4c^2}{2\zeta(2)}\log^2(T)  +E(T)- \frac{4c^2}{\zeta(2)^2} \log^2(T)  \\
		=& O(\log(T) + E(T))
	\end{align*}
	It remains to calculate the asymptotic growth of $E(T)$.
	\begin{align*}
		E(T) =& \frac{1}{\zeta(2)}\left(\left( 1- \frac{\lceil 2c \rceil}{ \lfloor cT+\frac{c}{T} \rfloor}\right) \log \left(\frac{T}{\lfloor cT+\frac{c}{T}\rfloor}\right)\right) + O\left(T^{-1}(\log T)^{\frac{2}{3}}(\log \log T)^\frac{4}{3}\right)\\
		&+ \int_{\lceil 2c \rceil}^{\lfloor cT+\frac{c}{T} \rfloor}\Bigg[ O\left( \frac{\log\left(\frac{T}{t}\right)}{t^2} \right) + O\left(t^{-2} \left(\log t \right)^\frac{2}{3} \left(\log \log t\right)^\frac{4}{3} \right) \\ &\qquad\qquad\qquad+ O\left( t^{-2} \left(\log t\right)^\frac{5}{3} \left( \log \log t\right)^\frac{4}{3} \right)+ O\left( t^{-2}\right)\Bigg]dt \\
		=& O(1) + O\left( \frac{1}{T}\right)+ O\left(T^{-1}(\log T)^{\frac{2}{3}}(\log \log T)^\frac{4}{3}\right) + O\left( \frac{\log T}{T} \right) + O\left(\frac{(\log T)^\frac{10}{3}}{T}\right) \\
		=& O(1)
	\end{align*}
	Therefore $O\left(E(T)\right) +O\left(\log(T)\right) = O\left(\log T\right)$. 
	
	Since $\int_{X^2} \left(\widetilde{\mathbbm{1}}_{P_{T^N,c}} - \frac{\textnormal{area}(P_{T^N,c})}{\zeta(2)}\right)^2\,d\mu=O(N)$, a similar argument to theorem \ref{main} yields, for any measurable $A\subset \mathbb{S}^0$,
	\begin{equation*}
		\frac{\widetilde{\mathbbm{1}}_{P_{\tau,c,A}}(\Lambda)}{\widetilde{\mathbbm{1}}_{P_{\tau,c}}(\Lambda)} = \text{vol}(A)  +o_{A,c,\Lambda,d}\big( (\log \tau)^{-\frac{1}{2}}(\log \log \tau)^\frac{3}{2}(\log\log\log\tau)^{\frac{1}{2}+\epsilon}\big)
	\end{equation*}
	\begin{remark}
		Let $\mathcal{H}_1(0)$ be the stratum of all unit area translation surfaces with one marked point. Then $\mathcal{H}_1(0)\cong \SL(2,\R)/\SL(2,\Z)$. Given a unimodular lattice $\Lambda \subset \C \cong \R^2$, we may find a translation surface $(X,\omega) \in \mathcal{H}_1(0)$ for which $\Lambda$ is its set of holonomy vectors. The number of saddle connections with holonomy vector lying in $P_{T,c}$ can then be counted using the result of section \ref{2d}, giving an error term for theorem 1.1 of \cite{athreyaErgodicTheoryDiophantine2016} in the case of the stratum $\mathcal{H}(0)$. See \cite{hasselblattHandbookDynamicalSystems2002,zorichFlatSurfaces2006a} for more details on translation surfaces.
	\end{remark}
	\section{Further Results}\label{further}
	\subsection{Effective Counting for Linear Forms}
	Recall the definition of $R_{T,c}$ from section \ref{SoR} and the calculation $\textnormal{vol}_d(R_{T,c}) = cB_mC_n\log T$.
	\begin{proof}[Proof of theorem \ref{effective2}]
		Define the flow $g_t:\R \rightarrow X_d$ by
		\begin{equation*}
			g_t=\begin{pmatrix}
				e^{\frac{n}{m}t} \text{Id}_m & 0 \\
				0 & e^{-t} \text{Id}_n
			\end{pmatrix} 
		\end{equation*}
		Then by Moore's Ergodicity Theorem \cite[Theorem 3]{mooreErgodicityFlowsHomogeneous1966}, this flow is ergodic on $X_d$. For $T\geq 1$ and for any $N\in\N$,
		\begin{equation*}
			\sum_{i=0}^{N-1} \widehat{\mathbbm{1}}_{R_{T,c}}\left( g_{\log T}^i \Lambda\right) = \widehat{\mathbbm{1}}_{R_{T^{N},c}}(\Lambda)
		\end{equation*}
		By lemma \ref{lemma}, 
		\begin{align*}
			\int_{X_{d}} \sum_{i=0}^{N-1} \widehat{\mathbbm{1}}_{R_{T,c}}\left( g_{\log T}^i \Lambda\right)\, d\mu &\leq \left( \int_{\R^d} \mathbbm{1}_{R_{T^{N},c}}(\mathbf{x}) d\mathbf{x}\right)^2 + O_{d}\left(\text{vol}_d(R_{T^{N},c})\right) \\
			&\leq c^2k^2B_m^2C_n^2 \log^2 T + O_d(\log T)
		\end{align*}
		Therefore,
		\begin{align*}
			&\int_{X_d} \left(\sum_{i=0}^{N-1} \widehat{\mathbbm{1}}_{R_{T,c}}\left( g_{\log T}^i \Lambda\right) - \text{vol}_{d}\left(R_{T^{N},c}\right)\right)^2 \,d\mu(\Lambda) \\
			=& \int_{X_d}\left( \widehat{\mathbbm{1}}_{R_{T^N,c}}\left( \Lambda\right)\right)^2\,d\mu(\Lambda) - 2\text{vol}_d(R_{T,c}) \int_{X_d}\widehat{\mathbbm{1}}_{R_{T^N,c}}\left(\Lambda\right)\,d\mu(\Lambda) + \text{vol}_d(R_{T^{N},c})^2 \\
			\leq& N^2c^2B_m^2C_n^2 \log^2 T + O_d(\log T^{N}) -  N^2c^2B_m^2C_n^2 \log^2 T  \\=& O_{T,d}(N) 
		\end{align*}
		The result follows by theorem \ref{Gaposhkin} similarly to the proof of theorem \ref{main}.
	\end{proof}
	\subsection{Effective Counting for Affine Lattices}
	To calculate the second moment of a function over the space of affine lattices, we use the following lemma.
	\begin{lemma}[{{\cite[Proposition 14]{el-bazDistributionDirectionsAffine2013}}}]\label{EBMV}
		Let $d\geq 2$ be a natural number and $F\in L^1(\R^{2d})$. Then
		\begin{equation*}
			\int_{Y_d}\sum_{\mathbf{v}\neq\mathbf{w}\in\Lambda+\mathbf{\xi}} F(\mathbf{v},\mathbf{w})d\nu(\Lambda,\nu) = \int_{\R^{2d}}F(\mathbf{x})d\mathbf{x}
		\end{equation*}
	\end{lemma}
	\begin{proof}[proof of theorem \ref{effective3}]
		Using theorems \ref{SMVT} and lemma \ref{EBMV}, we deduce that, for any $f\in L^1(\R^d)$,
		\begin{equation*}
			\int_{Y_d}\left(\widehat{f}(\Lambda+\xi) \right)^2d\nu(\Lambda,\mathbf{\xi}) = \left(\int_{\R^d}f(\mathbf{x})d\mathbf{x}\right)^2 + \int_{\R^d}f(\mathbf{x})^2d\mathbf{x}
		\end{equation*}
		By setting $f=\mathbbm{1}_B$, we have
		\begin{equation*}
			\int_{Y_d} \widehat{\mathbbm{1}}_B(\Lambda+\mathbf{\xi})d\nu(\Lambda,\mathbf{\xi}) = \left(\textnormal{vol}_d(B)\right)^2 + \textnormal{vol}_d(B)
		\end{equation*}
	The result then follows using a similar argument to theorem \ref{effective2}.
		\end{proof} 
	\bibliographystyle{abbrv}
	\typeout{}
	\bibliography{./Bibliography}
\end{document}